\documentclass[a4paper,twoside,10pt]{article}

\usepackage[a4paper,left=3cm,right=3cm, top=3cm, bottom=3cm]{geometry}
\usepackage[utf8]{inputenc}
\usepackage{graphicx,epstopdf}
\usepackage{amsmath}
\usepackage{amsthm}
\usepackage{amssymb}
\usepackage{color}
\usepackage[font=footnotesize,labelfont=bf]{caption}
\usepackage[author={Lorenzo}]{pdfcomment}

\usepackage{algorithm}
\restylefloat{table}
\theoremstyle{plain}
\newtheorem{theorem}{Theorem}[section]
\newtheorem{corollary}[theorem]{Corollary}
\newtheorem{lemma}[theorem]{Lemma}
\newtheorem{proposition}[theorem]{Proposition}
\theoremstyle{definition}

\theoremstyle{remark}
\newtheorem{remark}{Remark}

\definecolor{darkgreen}{rgb}{0.0, 0.55, 0.0}

\let\div\relax
\DeclareMathOperator{\div}{div}
\DeclareMathOperator{\dev}{\bf dev}
\DeclareMathOperator{\tr}{tr}
\DeclareMathOperator{\divbf}{\bf div}
\DeclareMathOperator{\gradbf}{\bf \nabla}
\newcommand{\h}{h}
\newcommand{\sigmabold}{\underline{\boldsymbol \sigma}}
\newcommand{\taubold}{\underline{\boldsymbol \tau}}
\newcommand{\sigmaboldh}{\underline{\boldsymbol \sigma}_\h}
\newcommand{\sigmaboldI}{\underline{\boldsymbol \sigma}_I}
\newcommand{\tauboldh}{\underline{\boldsymbol \tau}_\h}

\newcommand{\wbfstar}{\mathbf w^*}
\newcommand{\zbfstar}{\mathbf z^*}
\newcommand{\E}{K}
\newcommand{\hE}{\h_\E}
\newcommand{\F}{F}
\newcommand{\hF}{\h_\F}
\newcommand{\e}{e}
\newcommand{\he}{\h_\e}
\newcommand{\FcalE}{\mathcal F^\E}
\newcommand{\Nbb}{\mathbb N}
\let\P\relax
\DeclareMathOperator{\P}{P}
\newcommand{\Pbf}{\mathbf P}
\newcommand{\Pbb}{\mathbb P}
\newcommand{\Rbb}{\mathbb R}
\newcommand{\RM}{\mathbf{RM}}
\newcommand{\RMperp}{\RM^\perp}
\newcommand{\qbf}{\mathbf q}
\newcommand{\qbftilde}{\mathbf q}
\newcommand{\nbf}{\mathbf n}
\newcommand{\nbfE}{\nbf^\E}
\newcommand{\Norm}[1]{{\left\|{#1} \right\|}}
\newcommand{\SemiNorm}[1]{{\left|{#1} \right|}}
\newcommand{\SE}{S^\E}
\newcommand{\StildeE}{\widetilde S^\E}
\newcommand{\Sigmaboldh}{\underline{\boldsymbol \Sigma}_\h} 
\newcommand{\SigmaboldhE}{\Sigmaboldh(\E)}

\newcommand{\zerobf}{\mathbf 0}

\newcommand{\PiRM}{\boldsymbol \Pi_{\RM}}
\newcommand{\PiRMperp}{\boldsymbol \Pi_{\RM}^{\perp}}

\newcommand{\rbf}{\mathbf r}
\newcommand{\xbf}{\mathbf x}
\newcommand{\xbfE}{\xbf_\E}
\newcommand{\boldalpha}{\boldsymbol \alpha}
\newcommand{\boldbeta}{\boldsymbol \beta}
\newcommand{\boldlambda}{\boldsymbol \lambda}
\newcommand{\omegabold}{\boldsymbol \omega}
\newcommand{\p}{p}
\newcommand{\qbfp}{\qbf_\p}
\newcommand{\qbfunderp}{\underline{\qbf}_\p}
\newcommand{\qbftildep}{\qbftilde_\p}
\newcommand{\qbftildepF}{\qbftildep^\F}
\newcommand{\qbfppo}{\qbf_{\p+1}}
\newcommand{\qbfpperp}{\qbf_\p^\perp}
\newcommand{\mboldalpha}{\mathbf m_{\boldalpha}}
\newcommand{\mboldbeta}{\mathbf m_{\boldbeta}}
\newcommand{\mboldalphaF}{\mboldalpha^\F}
\newcommand{\mboldbetaF}{\mboldbeta^\F}
\newcommand{\mboldalphaperp}{\mboldalpha^\perp}
\newcommand{\mboldbetaperp}{\mboldbeta^\perp}
\DeclareMathOperator{\DoFbf}{\textbf{dof}}
\DeclareMathOperator{\DoFbfdiv}{\textbf {dof}^{\perp}}
\DeclareMathOperator{\DoFbfdivboldbeta}{\textbf {dof}^{\perp}_{\boldbeta}}
\newcommand{\Cbb}{\mathbb C}
\newcommand{\Dbb}{\mathbb D}
\newcommand{\muperp}{\mu^\perp}
\newcommand{\boldmu}{\boldsymbol \mu}
\newcommand{\taun}{\mathcal T_n}

\newcommand{\qbfpF}{\qbf_\p^\F}
\newcommand{\qbfpRM}{\qbf_p^{\text{RM}}}
\newcommand{\PizE}{\boldsymbol \Pi_\p^{0,\E}}
\newcommand{\TpE}{\mathbb T_\p(\E)}
\newcommand{\PiTpE}{\underline{\boldsymbol\Pi}_\p^{T}}
\newcommand{\qbbpT}{\underline{\mathbf q}_\p^{\mathbf T}}
\newcommand{\aE}{a^\E}
\newcommand{\ahE}{a_h^\E}
\newcommand{\Sigmabold}{\underline{\boldsymbol\Sigma}}
\newcommand{\Sigmaboldtilde}{\widetilde{\Sigmabold}}
\newcommand{\SigmaboldtildeE}{\Sigmaboldtilde(\E)}
\newcommand{\wbf}{\mathbf w}
\newcommand{\zbf}{\mathbf z}
\newcommand{\rbfppo}{\underline{\mathbf r}_{\p+1}}
\newcommand{\fbf}{\mathbf f}
\newcommand{\ubf}{\mathbf u}
\newcommand{\vbf}{\mathbf v}
\newcommand{\Vbf}{\mathbf V}
\newcommand{\Ibb}{\mathbb I}
\newcommand{\PizF}{\boldsymbol\Pi^{0,\F}_{\p}}

\newcommand{\nablaS}{\nabla_S}
\newcommand{\nablaSS}{\nabla_{SS}}
\newcommand{\betastarz}{\beta^*_0}
\newcommand{\Scal}{\mathcal S}
\newcommand{\HScal}{\mathbb H_{\Scal}}

\newcommand{\EcalF}{\mathcal E^\F}
\newcommand{\sigmaboldN}{\boldsymbol\sigma_N}
\newcommand{\ufrak}{\mathbf{\mathfrak u}}
\newcommand{\pfrak}{\mathfrak p}
\newcommand{\vfrak}{\mathbf{\mathfrak v}}
\newcommand{\qfrak}{\mathfrak q}
\newcommand{\Vfrak}{\mathbf{\mathfrak V}}
\newcommand{\Qfrak}{\mathfrak Q}
\newcommand{\Cafrak}{C_{\mathfrak a}}

\newcommand{\hOmega}{\h_\Omega}
\newcommand{\rhoOmega}{\rho_\Omega}
\newcommand{\cbf}{\mathbf c}
\newcommand{\sigmaboldhat}{\widehat{\sigmabold}}
\newcommand{\Htildebfo}{\widetilde{\Hbf}^1}

\newcommand{\Abf}{\mathbf A}
\newcommand{\Bbf}{\mathbf B}

\newcommand{\varphibold}{\boldsymbol \varphi}

\newcommand{\Hbf}{\mathbf H}
\newcommand{\Hbb}{\mathbb H}
\newcommand{\Lbf}{\mathbf L}
\newcommand{\Lbb}{\mathbb L}

\newcommand{\sizeMesh}{0.15}
\title{Stability and interpolation estimates of Hellinger--Reissner virtual element spaces}
\author{\footnotesize Michele Botti\thanks{MOX, Department of Mathematics, Politecnico di Milano, 20133 Milano, Italy (michele.botti@polimi.it, michele.visinoni@polimi.it)},
Lorenzo Mascotto\thanks{Department of Mathematics and Applications, University of Milano--Bicocca, 20125 Milan, Italy;
Faculty of Mathematics, University of Vienna, 1090 Vienna, Austria;
IMATI-CNR, 27100 Pavia, Italy
(lorenzo.mascotto@unimib.it)},\;
Giuseppe Vacca\thanks{Department of Mathematics, University of Bari, 70125 Bari (giuseppe.vacca@uniba.it)},\;
Michele Visinoni\footnotemark[1]}
\date{}

\begin{document}
%%%%%%%%%%%%%%%%%%%%%%%%%%%%%%%%%%%%%
\maketitle

\begin{abstract}
\noindent
We prove stability and interpolation estimates
for Hellinger--Reissner virtual elements;
the constants appearing in such estimates only depend on
the aspect ratio of the polytope under consideration
and the degree of accuracy of the scheme.
We further investigate numerically the behaviour
of the constants appearing in the stability estimates
on sequences of badly-shaped polytopes
and for increasing degree of accuracy.


\end{abstract}

%%%%%%%%%%%%%%%%%%%%%%%%%%%%%%%%%%%%%%%%%%%%%%%%%%%%%%%%%%%%%%%%%%%%%%%%%%
\section{Introduction} \label{section:introduction}
%%%%%%%%%%%%%%%%%%%%%%%%%%%%%%%%%%%%%%%%%%%%%%%%%%%%%%%%%%%%%%%%%%%%%%%%%%

\paragraph*{State-of-the-art.}
Virtual elements were introduced more than
a decade ago~\cite{Beirao-Brezzi-Cangiani-Manzini-Marini-Russo:2013}
as a generalization of finite elements
able to handle meshes of polytopic elements;
different types of elements have been designed for several problems,
including the linear elasticity problem
based on the Hellinger--Reissner (HR) principle.
A lowest order HR virtual element in 2D was introduced in~\cite{Artioli-DeMiranda-Lovadina-Patruno:2017}
and generalized to the arbitrary order case in~\cite{Artioli-DeMiranda-Lovadina-Patruno:2018};
the 3D version of the method is presented
in~\cite{Dassi-Lovadina-Visinoni:2020, Visinoni:2024}.
HR virtual elements consist of strongly symmetric stresses,
whose degrees of freedom are associated
with the interior of a $d$ dimensional element
and its $d-1$ dimensional facets, $d=2,3$,
and piecewise polynomial displacements.

The literature of finite elements involving strongly symmetric
stresses traces back to the end of the 60ies:
Watwood and Hartz~\cite{Watwood-Hartz:1968} constructed
lowest order spaces of strongly symmetric stresses
on triangular meshes;
therein, there are no vertex degrees of freedom
and the design of the element hinges on
a suitable split of each triangle into three smaller triangles.
That method was analyzed
by Hlav{\'a}{\v{c}}ek~\cite{Hlavecek:1979},
and Johnson and Mercier~\cite{Johnson-Mercier:1978}.
Henceforth, we shall be referring
to that kind of element as JM element.
The JM element was generalized to the 3D case by Krizek~\cite{Krizek:1982}.
A different avenue was undertaken decades later
by Arnold and Winther~\cite{Arnold-Winther:2002} (AW)
for triangular meshes;
no elemental splits are there necessary, the price to pay
being an increase of the dimension of local spaces
and the use of vertex degrees of freedom.
The AW original approach was generalized to rectangular meshes
in \cite{Arnold-Awanou:2005, Chen-Wang:2011}
and 3D elements in \cite{Adams-Cockburn:2005,Arnold-Awanou-Winther:2008}.
We also mention the work by Hu and Zhang \cite{Hu-Zhang:2015}
for polynomial order larger than $4$ on tetrahedral meshes,
which was employed by Chen and Huang
to construct a full elasticity complex in \cite{Chen-Huang:2022}.
Other elasticity complexes based on Alfeld and Worsey-Farin splits
are discussed in \cite{Christiansen-Gopalakrishnan-Guzman-Hu:2024, Gong:Gopalakrishnan-Guzman-Neilan:2023}.
More recently, JM elements in any dimension
were designed in~\cite{Gopalakrishnan-Guzman-Lee:2024}.

To the best of our understanding,
the above references can be clustered into two main categories:
AW types elements, whose design does not require
any split of the mesh elements
but in $d$ dimensions is based on degrees of freedom
on all geometrical entities with finite nonzero 
$\ell$-dimensional Hausdorff measure for all $\ell=0,\dots,d$;
JM types elements, which compared to AW elements
need fewer degrees of freedom
attached only to geometrical entities with codimension~$0$ and~$1$,
but require some elemental splits
for their design.

HR virtual element spaces retain advantages of AW and JM finite elements:
no split of the mesh elements is necessary for their design;
the degrees of freedom are only attached to the interior
and the facets of an element;
they support meshes of polytopic elements.

These upsides come at the price that discrete functions
are not known in closed form,
but only through certain polynomial projections
on the skeleton of the mesh
and the interior of each mesh element.
As such, the bilinear forms appearing in the weak formulation
of the linear elasticity problem in mixed formulation
are not computable for HR virtual element tensors.
For this reason, following the standard virtual element gospel~\cite{Beirao-Brezzi-Cangiani-Manzini-Marini-Russo:2013},
bilinear forms are discretized so as to be the sum of two terms:
one guaranteeing the polynomial consistency of the method
and requiring a projection of the HR virtual element tensors onto polynomial tensors;
the other providing stability, i.e., well-posedness
of the linear system equivalent to the method.

\paragraph*{Goals.}
In the literature on HR virtual elements
\cite{Artioli-DeMiranda-Lovadina-Patruno:2017, Artioli-DeMiranda-Lovadina-Patruno:2018, Dassi-Lovadina-Visinoni:2020, Visinoni:2024},
several stabilizations have been employed and heuristic arguments
as for its scaling with respect to local $L^2$ inner products are discussed.
In this paper, we show rigorous results along this directions:
for standard geometries, we exhibit stabilizations with correct scaling properties,
and stability constants in the sense of~\eqref{stability-bounds} below,
which only depend on the shape-regularity of the mesh
and the degree of accuracy of the scheme.
The technical tools we employ boil down to integration by parts,
polynomial inverse inequalities,
direct estimates, polynomial approximation results,
and well-posedness results of mixed formulation
with certain boundary conditions.

Another important theoretical aspect in the analysis of virtual elements
is the derivation of interpolation estimates,
with constants that are explicit with respect
to the shape of the polytopic elements.
Such estimates, yet with implicit constants,
are already available for the 2D lowest order case~\cite{Artioli-DeMiranda-Lovadina-Patruno:2017}
and they can be generalized to the general order and 3D cases
as remarked in~\cite{Artioli-DeMiranda-Lovadina-Patruno:2018, Dassi-Lovadina-Visinoni:2020, Visinoni:2024}.
Here, we provide a new proof of interpolation estimates
in HR virtual elements (with explicit constants)
based on using the stability estimates discussed above.

As we derive stability and interpolation estimates
for standard geometries (star-shaped elements
with no small facets) and for a fixed degree of accuracy,
the question arises naturally whether the proven bounds
are effectively robust with those respects.
Thus, we investigate numerically
the behaviour of the stability constants
on sequences of badly-shaped elements
and for increasing degree of accuracy.

\paragraph*{Notation.}
Given a domain~$D$ in~$\Rbb^d$, $d=2,3$,
with boundary $\partial D$,
outward unit normal vector~$\nbf_D$,
and diameter~$\h_D$,
$L^2(D)$ is the space of square integrable functions over~$D$
and~$L^2_0(D)$ is its subspace of functions with zero average over~$D$.
$H^s(D)$ denotes the Sobolev space of order~$s$ in~$\Nbb$;
we consider fractional order spaces defined by interpolation theory.
$H^1_0(D)$ denotes the subspace of~$H^1(D)$ consisting of functions with zero
trace over the boundary of~$D$.
We endow each Sobolev space with seminorm, norm, and bilinear forms
\[
\SemiNorm{\cdot}_{s,D},
\qquad\qquad\qquad
\Norm{\cdot}_{s,D},
\qquad\qquad\qquad
(\cdot,\cdot)_{s,D}.
\]
The negative order Sobolev space $H^{-1}(D)$
are defined as the dual space of~$H^1_0(D)$
and endowed with the norm
\begin{equation} \label{negative-norm}
\Norm{v}_{-1,D}:=
\sup_{\phi \in H^1_0(D),\; \SemiNorm{\phi}_{1,D} \ne 0}
                \frac{_{-1}\langle v,\phi\rangle_{1,D}}{\SemiNorm{\phi}_{1,D}}
    \qquad\qquad   \forall v \in H^{-1}(D) .
\end{equation}
Vector and tensor valued Sobolev and Lebesgue spaces
are defined similarly,
and are denoted by replacing $H$ and~$L$
by $\Hbf$ and $\Hbb$, and $\Lbf$ and $\Lbb$;
vector fields are highlighted in boldface font,
and tensors are further underlined
(e.g., $v$, $\vbf$, and $\underline {\vbf}$).
On most occasions, Roman and Greek letters
will be employed for vectors and tensors, respectively.

We usual standard notation for differential operators.
In particular, $\nablaS \cdot$ and~$\divbf \cdot$
are the symmetric gradient of a vector field
and the divergence operator of a tensor.
We further define the space~$\HScal^s(\divbf,D)$, $s\ge 0$,
of $\Hbb^s(D)$ symmetric tensors~$\sigmabold$
such that~$\divbf \sigmabold$ is in~$\Hbf^s(D)$.
We simply write $\Hbb(\divbf,D)$ if~$s=0$.
For tensors~$\sigmabold$ in $\HScal^s(\divbf,D)$,
the image space of the map (called trace operator) that associates $\sigmabold$
with $\sigmabold_{|\partial D}\ \nbf_D$ is called $\Hbf^{-\frac12}(\partial D)$,
which we endow with the norm
\[
\Norm{\sigmabold \nbf_D}_{-\frac12,\partial D}
:= \sup_{\vbf \in \Hbb^1(D),\; \Norm{\vbf}_{1,D}\ne 0}
    \frac{(\sigmabold,\nabla\vbf)_{0,D} 
            + (\divbf \sigmabold, \vbf)_{0,D}}{\h_D^{-\frac12}\Norm{\vbf}_{0,\partial D}
            +\SemiNorm{\vbf}_{\frac12,\partial D}}.
\]
The trace operator above is surjective,
see, e.g., \cite[Lemma 2.1.2]{Boffi-Brezzi-Fortin:2013},
and continuous:
there exists a positive~$c_{tr}$ only depending on the shape of~$D$ such that
\begin{equation} \label{trace:Hdiv}
\Norm{\sigmabold\nbf_D}_{-\frac12,\partial D}
\le c_{tr}
    (\h_D^{-1} \Norm{\sigmabold}_{0,D}
     + \Norm{\divbf \sigmabold}_{0,D})
\qquad\qquad \forall \sigmabold \in \HScal^s(\divbf,D).
\end{equation}
The constant~$c_{tr}$ coincides with the constant
appearing in the right-inverse trace inequality;
see, e.g., \cite[Theorem~3.37]{McLean:2000}.
We denote the duality pairing between $H^{-\frac12}(\partial D)$
and $H^{\frac12}(\partial D)$
by $\langle \cdot, \cdot\rangle$.

Given~$\Ibb$ the identity tensor,
we introduce~$\dev \sigmabold = \sigmabold - d^{-1} \tr(\sigmabold)\Ibb$.
$\P_\p(D)$ represents the space of polynomials of maximum degree~$\p$ in~$\Nbb$ over~$D$.
We use the convention $\P_{-1}(D) = \{ 0 \}$.
Vector and tensor polynomial spaces
are denoted by replacing~$\P$
with~$\Pbf$ and~$\Pbb$.
The space of rigid body motions
\[
\RM(D):= 
\{ \rbf(\xbf) = \boldalpha + \omegabold \wedge \xbf \mid
                \boldalpha\in\Rbb^3,\  \omegabold \in \Rbb^3     \}
\]
has dimension six and is spanned by
\[
(1,0,0);            \qquad
(0,1,0);            \qquad
(0,0,1);            \qquad
(1,0,0)\wedge\xbf;  \qquad
(0,1,0)\wedge\xbf;  \qquad
(0,0,1)\wedge\xbf .
\]
For positive constants~$a$ and~$b$,
we write~$a \lesssim b$ if there exists a uniform positive constant~$c$
such that~$a\le c \ b$.
If~$a\lesssim b$ and~$b \lesssim a$, we write~$a \approx b$.
On relevant occasions, we shall pinpoint
the actual dependence of the hidden constant.

\paragraph*{The model problem.}
In what follows, we focus on 3D problems only,
albeit the forthcoming analysis
can be generalized to the 2D case with minor modifications.
Given a Lipschitz polyhedral domain~$\Omega$ in~$\Rbb^3$
with boundary~$\Gamma:=\partial \Omega$,
a symmetric, uniformly elliptic elasticity tensor
$\Dbb = \Cbb^{-1}:\Rbb^{3\times3}\to\Rbb^{3\times3}$
on the space of symmetric $3\times3$ matrices,
and~$\fbf$ in~$\Lbf^2(\Omega)$,
we consider the linear elasticity problem:
Find~$\ubf:\Omega\to\Rbb^3$ such that
\begin{equation} \label{strong-formulation}
\begin{cases}
     -\divbf \ \sigmabold = \fbf                & \text{in } \Omega\\
     \sigmabold = \Cbb\nablaS (\ubf) & \text{in } \Omega\\
     \sigmabold \nbf=\mathbf 0                 & \text{on } \partial \Omega.
\end{cases}
\end{equation}
In the standard isotropic elasticity case,
given nonnegative Lam\'e moduli~$\lambda\ge 0$ and~$\mu> 0$,
we have
\begin{equation} \label{Lame}
\Cbb \taubold:= 2 \mu \dev (\taubold)
                + \frac{2\mu+3\lambda}{3} \tr (\taubold)\Ibb,
\qquad\qquad
\Dbb \taubold:= \frac{\dev (\taubold)}{2\mu}
                + \frac{\tr (\taubold)\Ibb}{3(2\mu+3\lambda)}.
\end{equation}
We introduce the spaces and inner product 
\small{\begin{equation} \label{spaces&bfs}
\begin{split}
& \Vbf:= \{\vbf\in\Lbf^2(\Omega) \mid 
(\vbf,\rbf)_{0,\Omega} = 0 \quad \forall \rbf \in \RM(\Omega) \},\\
& \Sigmabold:=
\{ \taubold \in \HScal(\divbf,\Omega) \mid
\langle \taubold \ \nbf , \vbf \rangle = 0
\quad \forall \vbf \in \Hbf^1(\Omega) \},
\qquad a(\sigmabold,\taubold)
:= (\Dbb \sigmabold,\taubold)_{0,\Omega}
\quad \forall \sigmabold, \taubold \in \Sigmabold.
\end{split}
\end{equation}}\normalsize
We endow the space~$\Vbf$ with~$\Norm{\cdot}_{\Vbf}$, which is the usual~$L^2$ norm,
and~$\Sigmabold$ with the norm given by
\[
\Norm{\cdot}_{\Sigmabold}^2
:= \Norm{\cdot}_{0,\Omega}^2
    + \hOmega^2 \Norm{\div \cdot}^2_{0,\Omega}.
\]
The HR weak formulation of~\eqref{strong-formulation} is
\begin{equation} \label{weak-formulation}
\begin{cases}
\text{Find } (\sigmabold,\ubf) \in \Sigmabold \times \Vbf \text{ such that}\\
a(\sigmabold,\taubold) + (\divbf \taubold, \ubf)_{0,\Omega} = 0   & \forall \taubold \in \Sigmabold\\
(\divbf \sigmabold, \vbf)_{0,\Omega} = - (\fbf,\vbf)_{0,\Omega}       & \forall \vbf \in \Vbf.
\end{cases}
\end{equation}
Theorem~\ref{theorem:stability-HR} below guarantees that problem~\eqref{weak-formulation}
is well-posed even for inhomogeneous essential boundary conditions,
with a priori bounds
\[
\Norm{\sigmabold}_{\Sigmabold}
+ \Norm{\ubf}_{\Vbf}
\le C \Norm{\fbf}_{0,\Omega}.
\]
Above, $C$ is a positive constant, which only depends
on the ratio between the diameter of~$\Omega$
and the radius of the largest ball contained in~$\Omega$.
If other types of boundary conditions are considered,
then the well-posedness of the problem can be found, e.g.,
in \cite[Chapter~8]{Boffi-Brezzi-Fortin:2013};
in this case, we are not able to provide an explicit dependence for the stability constant.

\paragraph*{Regular meshes of polytopic elements.}
Henceforth, we focus on three dimensional domains.
Even though the forthcoming analysis is only concerned with local stability and interpolation estimates,
we discuss regularity assumptions for given meshes of polytopic elements,
as the results we shall discuss
can be used in the analysis of virtual elements
for problem~\eqref{weak-formulation} below.

Given~$X$ either an edge, facet, or polyhedron,
we introduce its diameter~$h_X$, centroid~$\xbf_X$, and measure~$\vert X \vert$.
Given~$\E$ a polyhedron, its set of facets is~$\FcalE$
and the outward unit normal vector is~$\nbfE$.
Given $\F$ a facet of the mesh, its set of edges is $\EcalF$.

We consider sequences of regular meshes~$\{ \taun \}$
of polytopic elements in the following sense:
there exists a positive~$\rho$ such that, for all meshes~$\taun$ in the sequence,
\begin{itemize}
    \item each~$\E$ in~$\taun$ is star-shaped with respect to a ball of radius larger than or equal to~$\rho\hE$;
    \item for each~$\E$ in~$\taun$, each facet~$\F$ in~$\FcalE$
    is star-shaped with respect to a disk of radius larger than or equal to~$\rho \hF$,
    and is such that $\hF$ is larger than or equal to~$\rho\hE$;
    \item for each~$\E$ in~$\taun$ and~$\F$ in~$\FcalE$, every edge~$\e$ in~$\EcalF$
    is such that $\he$ is larger than or equal to~$\rho \hF$.
\end{itemize}
The above assumptions may be relaxed; yet,
we stick to the above setting
for the presentation's sake.
Henceforth, given the Lam\'e parameters~$\lambda$ and~$\mu$
as in~\eqref{Lame}, we assume that
\begin{equation} \label{D-C-constants}
\lambda \text{ and (consequently) } \mu \text{ in~\eqref{Lame}
are piecewise constant over any mesh in } \{\taun\}.
\end{equation}
For future convenience, we expand the bilinear form~$a(\cdot,\cdot)$
into a sum of local contribution over the elements
of a mesh~$\taun$:
\begin{equation} \label{splitting:bf-a}
    a(\sigmabold,\taubold) 
    =  \sum_{\E\in\taun} \aE(\sigmabold, \taubold)
    := \sum_{\E\in\taun} (\Dbb \sigmabold,\taubold)_{0,\E}.
\end{equation}

%%%%
\paragraph*{Outline of the paper.}
%%%%
In Section~\ref{section:VE}, we introduce HR type virtual elements
and discretize $L^2$-type inner products using polynomial projections
and stabilizing bilinear forms.
In Section~\ref{section:stability}, we discuss stability bounds,
which we use
in Section~\ref{section:interpolation}
to derive interpolation estimates in HR VE spaces.
We assess numerically the behaviour of the stabilizations
on sequences of badly-shaped elements
and for increasing degree of accuracy in Section~\ref{section:numerics}.
Conclusions are drawn in Section~\ref{section:conclusions},
while technical results from the PDE theory
are investigated in Appendix~\ref{appendix}.

%%%%%%%%%%%%%%%%%%%%%%%%%%%%%%%%%%%%%%%%%%%%%%%%%%%%%%%%%%%%%%%%%%%%%%%%%%
\section{Hellinger--Reissner virtual element spaces and forms} \label{section:VE}
%%%%%%%%%%%%%%%%%%%%%%%%%%%%%%%%%%%%%%%%%%%%%%%%%%%%%%%%%%%%%%%%%%%%%%%%%%
We review the construction of HR virtual elements,
the choice of suitable degrees of freedom,
the definition of certain polynomial spaces and polynomial projections,
and the design of the discrete bilinear forms.

\paragraph*{Polynomial spaces.}
We introduce the space of vector polynomials
that are orthogonal in $L^2(\E)$ to rigid body motions:
\[
\RMperp_\p(\E) :=
\{ \qbfp \in \Pbf_\p(\E) \mid
    (\qbfp,\rbf)_{0,\E} = 0 \quad \forall \rbf \in \RM(\E) \}
\]
and note that
\begin{equation} \label{orthogonal-splitting}
\Pbf_\p(\E)
= \RM(\E) \oplus_{\Lbf^2(\E)} \RMperp_\p(\E).
\end{equation}
For all elements~$\E$ in~$\taun$, we further introduce the space
\begin{equation} \label{TpE}
\TpE:= \{ \Cbb\nablaS (\qbfppo)
            \mid \qbfppo \in \Pbf_{\p+1}(\E)         \}.
\end{equation}

\paragraph*{Virtual element spaces.}
Given an element~$\E$ of~$\taun$ and~$\p$ in~$\Nbb$,
we define the Hellinger--Reissner virtual element space
\[
\SigmaboldhE:=
\{ \sigmaboldh \in \HScal(\divbf,\E) \mid  \sigmaboldh 
        \text{ solves weakly } \eqref{local-functions-strong} \text{ below} \},
\]
where
\begin{equation} \label{local-functions-strong}
\begin{cases}
\sigmaboldh = \Cbb \nablaS (\wbfstar)\\
 \divbf \sigmaboldh = \qbf_\p \in \Pbf_\p(\E)
        & \qquad \text{for some vector field } \wbfstar. \\
\sigmaboldh\nbfE{}_{|\F} = \qbftildepF \in \Pbf_\p(\F) \quad \forall \F \in \FcalE
\end{cases}
\end{equation}
Given $\aE(\cdot,\cdot)$ as in~\eqref{splitting:bf-a},
the discrete tensor~$\sigmaboldh$ defined in~\eqref{local-functions-strong} solves
\begin{equation} \label{local-functions-weak}
\begin{cases}
\text{Find } \sigmaboldh \in \HScal(\divbf,\E),\ \wbfstar \in \Lbf^2(\E) \text{ such that} \\
\aE(\sigmaboldh, \taubold) + (\divbf \taubold, \wbfstar)_{0,\E} = 0   & \forall \taubold \in \HScal(\divbf,\E) \\
(\divbf \sigmaboldh, \zbfstar)_{0,\E} = (\qbf_\p, \zbfstar)_{0,\E}          & \forall \zbfstar \in \Lbf^2(\E)\\
\sigmaboldh\nbfE {}_{|\F} = \qbftildepF                                  & \forall \F \in \FcalE .
\end{cases}
\end{equation}
Since the displacement~$\wbfstar$ in~\eqref{local-functions-weak}
is unique up to rigid body motions,
we fix it so as to be
$L^2$-orthogonal to~$\RM(\E)$.

Since~$\Cbb$ is a constant tensor over~$\E$, see~\eqref{D-C-constants},
the space~$\TpE$ defined in~\eqref{TpE} only contains polynomials up to order~$\p$.
Therefore, we have the inclusion
\begin{equation} \label{polynomials-in-space}
    \TpE \subset \SigmaboldhE.
\end{equation}

\begin{remark}
In principle, a lowest order element can be designed as well.
In particular, no internal degrees of freedom would be required
and slightly different boundary conditions would be imposed in~\eqref{local-functions-strong};
we refer to~\cite{Artioli-DeMiranda-Lovadina-Patruno:2017, Dassi-Lovadina-Visinoni:2020}
for more details.
The forthcoming analysis generalizes straightforwardly also to this setting,
whence we shall skip the details
for the presentation's sake.
\end{remark}

\paragraph*{Degrees of freedom.}
Hereafter, by scaled polynomial fields we mean polynomial fields~$\qbfp$
with $\Norm{\qbfp}_{L^\infty(X)}=1$, $X=\E$ or~$\F$.
Given~$\tauboldh$ in~$ \HScal(\divbf,\E)$,
we consider the following linear functionals:
\begin{itemize}
    \item for all facets~$\F$ of~$\E$, given a basis~$\{ \mboldalphaF \}$
    of~$\Pbf_\p(\F)$ of scaled polynomial fields
    that are shifted with respect to~$\xbf_\F$,
    the scaled vector moments of the tractions on~$\F$:
    \begin{equation} \label{face-moments}
        \frac{1}{\vert\F\vert} \int_{\F} \tauboldh\nbfE \cdot \mboldalphaF ;
    \end{equation}
    \item given a basis~$\{ \mboldalphaperp \}$
    of~$\RMperp_\p(\E)$ of scaled polynomial fields
    that are shifted with respect to~$\xbfE$,
    the interior scaled vector moments of the divergence:
    \begin{equation} \label{interior-moments}
        \frac{\hE}{\vert\E\vert} \int_{\E} \divbf\tauboldh \cdot \mboldalphaperp.
    \end{equation}
\end{itemize}
The above functionals form a set of unisolvent degrees of freedom for~$\SigmaboldhE$;
see \cite[Remark~1]{Visinoni:2024}.

%%%%%
\paragraph*{Polynomial projectors.}
We introduce several projectors.
For each element~$\E$,
the first one is $\PiRM: \Lbf^2(\E) \to \RM(\E)$ given by
\begin{equation} \label{RM-L2-projection}
(\vbf -\PiRM \vbf, \qbfpRM)_{0,\E} = 0
\qquad \qquad \forall \vbf \in \Lbf^2(\E) ,\; \qbfpRM \in \RM(\E) .
\end{equation}
Next, we define $\PiRMperp: \Lbf^2(\E) \to \RMperp(\E)$ as
\begin{equation} \label{RM-perp-L2-projection}
(\vbf - \PiRMperp \vbf, \qbfpperp)_{0,\E}=0
\qquad \qquad \forall \vbf \in \Lbf^2(\E),
\quad \qbfpperp \in \RMperp(\E).
\end{equation}

%%%
\begin{remark} \label{remark:computability-divergence}
For~$\tauboldh$ in~$\SigmaboldhE$,
$\PiRM \divbf \tauboldh$ can be computed for~$\tauboldh$ in~$\SigmaboldhE$
using the boundary degrees of freedom~\eqref{face-moments};
$\PiRMperp \divbf \tauboldh$
using the interior degrees of freedom~\eqref{interior-moments}.
Therefore, \eqref{orthogonal-splitting} implies that
$\divbf\tauboldh$ is available in closed-form.
\end{remark}

Additionally, we define~$\PiTpE: \Lbb^2(\E) \to \TpE$ as
\begin{equation} \label{definition:PiT}
(\Dbb(\taubold - \PiTpE \taubold), \qbbpT)_{0,\E}
= 0
\qquad\qquad\qquad\qquad \forall \qbbpT \in \TpE.
\end{equation}
or equivalently
\[
(\taubold - \PiTpE \taubold, \nablaS(\qbfppo))_{0,\E}
= 0
\qquad\qquad\qquad\qquad \forall \qbfppo \in \Pbf_{\p+1}(\E).
\]
The operator~$\PiTpE$ is computable for functions in~$\SigmaboldhE$
using the degrees of freedom \eqref{face-moments} and~\eqref{interior-moments}.
In fact, as observed for the 2D case
in~\cite[eq. (26)]{Artioli-DeMiranda-Lovadina-Patruno:2018},
the orthogonality condition in~\eqref{definition:PiT}
for $\taubold$ replaced by $\tauboldh$ in $\SigmaboldhE$
is equivalent to find~$\rbfppo$ in~$\Pbf_{\p+1}(\E)$ such that
\[
(\Cbb\nablaS(\rbfppo),
    \nablaS (\qbfppo))_{0,\E}
= (\tauboldh, \nablaS(\qbfppo))_{0,\E}
\qquad \forall \qbfppo \in \Pbf_{\p+1}(\E) .
\]
The right-hand side above is computable
via the degrees of freedom~\eqref{face-moments} and~\eqref{interior-moments}
after an integration by parts.

Finally, on each facet~$\F$, define~$\PizF: \Lbf^2(\F) \to \Pbf_\p(\F)$ as
\begin{equation}\label{L2-projector-face}
    (\vbf - \PizF \vbf, \qbfpF)_{0,\F} = 0
    \qquad \qquad \forall \vbf \in \Lbf^2(\F),
    \qquad \forall \qbfpF \in \Pbf_\p(\F).
\end{equation}

\paragraph*{Discrete bilinear forms and stabilizations.}
In the HR virtual element method,
the local $L^2$ inner product $\aE(\sigmabold,\taubold)$
is discretized using two ingredients.
The first one is the projector~$\PiTpE$ in~\eqref{definition:PiT};
the second one is a bilinear
form~$\SE(\cdot,\cdot): \SigmaboldhE \times \SigmaboldhE \to \Rbb$
satisfying two properties:
\begin{itemize}
    \item $\SE(\cdot,\cdot)$ is computable only using
    the degrees of freedom~\eqref{face-moments} and~\eqref{interior-moments};
    \item $\SE(\cdot,\cdot)$ ``scales'' like $\Norm{\cdot}_{0,\E}^2$ on the kernel of~$\PiTpE$,
    i.e., there exist positive constants~$\alpha_* \le \alpha^*$
    independent of~$\lambda$, $\mu$, and~$\hE$ such that
    \begin{equation} \label{stability-bounds}
    \mu^{-1}_{|\E} \alpha_* \Norm{\tauboldh}_{0,\E}^2
    \le \SE(\tauboldh,\tauboldh)
    \le \mu^{-1}_{|\E} \alpha^* \Norm{\tauboldh}_{0,\E}^2
    \qquad\qquad \forall \tauboldh \in \SigmaboldhE \cap \ker (\PiTpE).
    \end{equation}
    The constants~$\alpha_*$ and~$\alpha^*$ may depend
    on the regularity parameter~$\rho$ in Section~\ref{section:introduction}
    and~$\p$ in~\eqref{local-functions-strong}.
\end{itemize}
\begin{remark} \label{remark:true-stab-bounds}
For more standard virtual elements
(such as those for the approximation of the Poisson problem,
the Stokes equations, edge, and face virtual elements \dots),
the bounds in \eqref{stability-bounds} can be made sharper
\cite[Corollary~2]{Mascotto:2023}:
the lower bound is typically proven for functions (fields, tensors, \dots)
in the virtual element space;
the upper bound for functions (fields, tensors, \dots)
in a Sobolev space with extra zero average conditions.
In the HR virtual element setting, this is not possible;
see the proof of Proposition~\ref{proposition:equivalence-two-stab} below.
\end{remark}

Let~$\Ibb$ be the identity operator.
With the two ingredients above at hand,
we introduce the local discrete bilinear forms
\begin{equation} \label{discrete:bf}
\ahE(\sigmaboldh,\tauboldh)
:=  \aE(\PiTpE \sigmaboldh, \PiTpE \tauboldh)
   + \SE( (\Ibb -\PiTpE)\sigmaboldh, (\Ibb-\PiTpE) \tauboldh ) 
   \qquad\qquad \forall\E\in\taun.
\end{equation}
In Section~\ref{section:stability} below,
we shall exhibit an explicit choice of~$\SE(\cdot,\cdot)$
and prove the corresponding stability bounds~\eqref{stability-bounds}.

%%%%%%%%%%%%%%%%%%%%%%%%%%%%%%%%%%%%%%%%%%%%%%%%%%%%%%%%%%%%%%%%%%%%%%%%%%
\section{Explicit stabilizations and stability bounds} \label{section:stability}
%%%%%%%%%%%%%%%%%%%%%%%%%%%%%%%%%%%%%%%%%%%%%%%%%%%%%%%%%%%%%%%%%%%%%%%%%%
We introduce two explicit stabilizations
and prove stability bounds as in~\eqref{stability-bounds}.
More precisely, we investigate
a ``projection-based'' stabilization in Section~\ref{subsection:projection-based-stab}
and a ``dofi-dofi'' stabilization in Section~\ref{subsection:dofi-dofi-stab}.

%%%%%%%%%%%%
\subsection{A ``projection-based'' stabilization} \label{subsection:projection-based-stab}
%%%%%%%%%%%%
Given~$\PiRMperp$ as in~\eqref{RM-perp-L2-projection},
for all $\E$ in $\taun$, and all $\sigmaboldh$ and~$\tauboldh$ in~$\SigmaboldhE$,
we define
\begin{equation} \label{stabilization:general-order}
    \SE(\sigmaboldh,\tauboldh)
    := \mu^{-1}_{|\E} \hE \sum_{\F \in \FcalE} (\sigmaboldh\nbfE, \tauboldh\nbfE)_{0,\F}
       + \mu^{-1}_{|\E}\hE^2 (\PiRMperp\divbf \sigmaboldh, \PiRMperp \divbf \tauboldh )_{0,\E}.
\end{equation}
A practical alternative option discussed in the literature
\cite{Artioli-DeMiranda-Lovadina-Patruno:2017, Artioli-DeMiranda-Lovadina-Patruno:2018}
is the stabilization above
replacing $\mu^{-1}_{|\E}$ by~$\tr(\Dbb)$.

\begin{proposition} \label{proposition:stability-HR-general-order}
The bilinear form~$\SE(\cdot,\cdot)$ defined in~\eqref{stabilization:general-order}
is computable by means of the degrees of
freedom~\eqref{face-moments} and~\eqref{interior-moments},
and satisfies the stability bounds~\eqref{stability-bounds}.
In fact, the result holds true also for tensors~$\tauboldh$
such that~$\PiTpE\tauboldh$ is different from the zero tensor.
\end{proposition}
\begin{proof}
The computability of the bilinear form follows from the computability
of~$\PiRMperp \divbf(\cdot)$ for functions in~$\SigmaboldhE$,
see Remark~\ref{remark:computability-divergence},
and the fact that tractions on facets are given polynomials,
see~\eqref{local-functions-strong}.
\medskip

We assume that~$\hE=1$.
The general assertion is a consequence of a scaling argument.
In what follows, let~$\tauboldh$ solve~\eqref{local-functions-strong}
(or equivalently~\eqref{local-functions-weak}).

Since the symmetric gradient of a rigid body motion is the zero tensor, we have
\begin{equation} \label{property:sym-grad&RM}
    \nablaS(\divbf \tauboldh)
    = \nablaS(\PiRMperp \divbf \tauboldh)
    \qquad\qquad\qquad \forall \tauboldh \in \SigmaboldhE.
\end{equation}

\noindent \paragraph*{The lower bound.}
Using Theorem~\ref{theorem:stability-HR} for the local problem~\eqref{local-functions-weak} 
on~$\E$, we deduce the existence of a positive constant $c_{ST}$
depending only on~$\rho$ in Section~\ref{section:introduction} such that
\begin{equation} \label{lower-bound:last-but-one-general-order}
\Norm{\tauboldh}_{0,\E}
\le \Norm{\tauboldh}_{\div,\E}
+ \Norm{\wbfstar}_{0,\E}
\le c_{ST} \left( \Norm{\divbf \tauboldh}_{0,\E}
+ \Norm{\tauboldh\nbfE}_{0,\partial \E} \right).
\end{equation}
Integrating by parts (IBP),
and using~\eqref{property:sym-grad&RM},
the Cauchy-Schwarz (CS) inequality,
the~$H^1-L^2$ polynomial inverse
inequality (II)~\cite{Verfurth:2013} (constant~$c_{inv}^A$),
the fact that~$\div \tauboldh$ is a vector polynomial,
an inverse $L^2$ trace inequality (ITI)~\cite{Verfurth:2013}
bounding the~$L^2$ norm on the boundary
by the $L^2$ norm in the interior (constant~$c_{inv}^B$),
and two Young's (Y) inequalities with corresponding
positive constants~$\varepsilon_1$ and~$\varepsilon_2$,
we write
\small{\[
\begin{split}
\Norm{\divbf \tauboldh}_{0,\E}^2
& = \int_\E \divbf \tauboldh \cdot \divbf \tauboldh
  \overset{\text{(IBP)}}{=}
  -\int_\E \tauboldh : \nablaS(\divbf \tauboldh)
    + \int_{\partial\E} \tauboldh\nbfE \cdot \divbf \tauboldh\\
& \overset{\eqref{property:sym-grad&RM}}{=}
    -\int_\E \tauboldh : \nablaS(\PiRMperp \divbf\tauboldh)
    + \int_{\partial\E} \tauboldh\nbfE \cdot \divbf \tauboldh\\
& \overset{\text{(CS)}}{\le}
    \Norm{\tauboldh}_{0,\E}
        \SemiNorm{\PiRMperp \divbf \tauboldh}_{1,\E}
        + \Norm{\tauboldh\nbfE}_{0,\partial\E}
      \Norm{\divbf \tauboldh}_{0,\partial\E}\\
&   \overset{\text{(II),\, (ITI)}}{\le} (c_{inv}^A+c_{inv}^B) 
    \Big( \Norm{\tauboldh}_{0,\E}
        \Norm{\PiRMperp \divbf \tauboldh}_{0,\E}
        + \Norm{\tauboldh\nbfE}_{0,\partial\E}
      \Norm{\divbf \tauboldh}_{0,\E} \Big)\\
&   \overset{\text{(Y)}}{=} \frac12 (c_{inv}^A+c_{inv}^B) 
        \Big(\varepsilon_1^2 \Norm{\tauboldh}_{0,\E}^2
        + \varepsilon_1^{-2} \Norm{\PiRMperp \divbf \tauboldh}_{0,\E}^2
        + \varepsilon_2^{-2} \Norm{\tauboldh\nbfE}_{0,\partial\E}^2
        + \varepsilon_2^2 \Norm{\divbf \tauboldh}_{0,\E}^2\Big).
\end{split}
\]}\normalsize
Being the constants~$c_{inv}^A$ and~$c_{inv}^B$ related to polynomial inverse inequalities on (regular sub-tessellation on) polytopes,
they depend on the regularity parameter~$\rho$ in Section~\ref{section:introduction}
and~$\p$ in~\eqref{local-functions-strong}.

Taking~$\varepsilon_2$ sufficiently small,
we have that there exists a positive~$c_{\varepsilon_2}$
only depending on~$\varepsilon_2$,
$c_{inv}^A$, and $c_{inv}^B$ such that
\[
\Norm{\divbf \tauboldh}_{0,\E}
\le c_{\varepsilon_2} \Big( \varepsilon_1 \Norm{\tauboldh}_{0,\E}
                + \varepsilon_1^{-1} \Norm{\PiRMperp \divbf \tauboldh}_{0,\E}
                + \Norm{\tauboldh\nbfE}_{0,\partial\E}
                \Big).
\]
Inserting this bound in~\eqref{lower-bound:last-but-one-general-order} and taking~$\varepsilon_1$ sufficiently small yields the assertion:
there exists a positive constant~$c_{\varepsilon_1}$
only depending on~$\varepsilon_1$ and~$c_{\varepsilon_2}$
such that
\[
\Norm{\tauboldh}_{0,\E}
\le  c_{ST} \ c_{\varepsilon_1}
\Big(\Norm{\tauboldh\nbfE}_{0,\partial\E} 
        +  \Norm{\PiRMperp \divbf \tauboldh}_{0,\E} \Big).
\]
Notably, the constant~$\alpha_*$ in~\eqref{stability-bounds}
for the stabilization $\SE(\cdot,\cdot)$ in~\eqref{stabilization:general-order}
only depends on~$c_{inv}^A$, $c_{inv}^B$, and~$c_{ST}$.

\noindent \paragraph*{The upper bound.}
Using a polynomial inverse inequality~\cite{Verfurth:2013}
on~$\partial \E$ (constant~$c_{inv}^C$) entails
\[
\Norm{\tauboldh\nbfE}_{0,\partial \E}
\le c_{inv}^C
\Norm{\tauboldh\nbfE}_{-\frac12,\partial \E}.
\]
Being the constant~$c_{inv}^C$ related to polynomial inverse inequalities
on the boundary of polytopes,
it depends on the regularity parameter~$\rho$ in Section~\ref{section:introduction}
and~$\p$ in~\eqref{local-functions-strong}.

Next, we use the $H(\div)$ trace inequality \eqref{trace:Hdiv}
(constant~$c_{tr}$)
and get
\[
\Norm{\tauboldh\nbfE}_{-\frac12,\partial \E}
\le c_{tr} \left(\Norm{\tauboldh}_{0,\E}
                  + \Norm{\divbf \tauboldh}_{0,\E} \right).
\]
It suffices to estimate the second term on the right-hand side.
Using the $L^2 - H^{-1}$ polynomial inverse inequality~\cite{Verfurth:2013}
(constant~$c_{inv}^D$),
definition~\eqref{negative-norm},
an integration by parts,
and the Cauchy-Schwarz inequality
yields
\small\[
\Norm{\PiRMperp \divbf \tauboldh}_{0,\E}
\le \Norm{\divbf \tauboldh}_{0,\E}
\le c_{inv}^D \Norm{\divbf \tauboldh}_{-1,\E}
=  c_{inv}^D \sup_{\vbf \in \Hbf^1_{\zerobf}(\E)}
    \frac{(\divbf \tauboldh, \vbf)_{0,\E}}{\SemiNorm{\vbf}_{1,\E}}
\le c_{inv}^D \Norm{\tauboldh}_{0,\E},
\]\normalsize
thereby implying the upper bound in~\eqref{stability-bounds}.
The constant~$\alpha^*$ in~\eqref{stability-bounds}
for the stabilization $\SE(\cdot,\cdot)$ in~\eqref{stabilization:general-order}
only depends on~$c_{inv}^C$, $c_{inv}^D$, and~$c_{tr}$.
\end{proof}
%%

%%%%%%%%%%%%
\subsection{A ``dofi-dofi'' stabilization} \label{subsection:dofi-dofi-stab}
%%%%%%%%%%%%
With an abuse of notation, we denote the set of the degrees
of freedom of~$\SigmaboldhE$ by~$\{ \DoFbf \}$,
which we split into the union of facet
$\{ \DoFbf^\F \}$, for all~$\F$ facets of~$\E$,
and the divergence $\{ \DoFbfdiv \}$ degrees of freedom.

Given the standard $\ell^2$ inner product
$(\cdot,\cdot)_{\ell^2}$ for sequences,
we introduce the ``dofi-dofi''
stabilization~$\StildeE :\SigmaboldhE \times \SigmaboldhE \to \Rbb$
defined on any element~$\E$ as
\begin{equation} \label{stabilization:dofi-dofi}
\begin{split}
    \StildeE(\sigmaboldh, \tauboldh)
    & := \mu^{-1}_{|\E} \hE^3 (\DoFbf(\sigmaboldh), \DoFbf(\tauboldh))_{\ell^2} \\
    & = \mu^{-1}_{|\E} \hE^3 \sum_{\F \in \FcalE} (\DoFbf^\F(\sigmaboldh), \DoFbf^\F(\tauboldh))_{\ell^2}
      + \mu^{-1}_{|\E} \hE^3 (\DoFbfdiv(\sigmaboldh), \DoFbfdiv(\tauboldh))_{\ell^2}.
\end{split}
\end{equation}
We prove that this stabilization is equivalent to the
``projection based'' stabilization
in~\eqref{stabilization:general-order}
under a suitable choice of the polynomial bases
appearing in~\eqref{face-moments} and~\eqref{interior-moments}.

%%%%%%
\begin{proposition}  \label{proposition:equivalence-two-stab}
Given~$\SE(\cdot,\cdot)$ and~$\StildeE(\cdot, \cdot)$ the stabilizations defined
in~\eqref{stabilization:general-order} and~\eqref{stabilization:dofi-dofi},
there exist positive constants $\beta_* \le \beta^*$ such that
\[
\beta_* \SE(\tauboldh,\tauboldh)
\le \StildeE(\tauboldh,\tauboldh)
\le \beta^* \SE(\tauboldh,\tauboldh)
\qquad \qquad
\forall \tauboldh \in \SigmaboldhE,\quad \forall \E \in \taun.
\]
The constants~$\beta_*$ and~$\beta^*$ may depend
on the regularity parameter~$\rho$ in Section~\ref{section:introduction}
and~$\p$ in~\eqref{local-functions-strong}.
\end{proposition}
%%%%%%
\begin{proof}
We show the lower and upper bound separately.

\paragraph*{The lower bound (part~1).}
We first show, for all facets~$\F$ in~$\FcalE$, the existence of a positive constant~$c_1$ such that
\begin{equation} \label{equivalence-lower-bound-face}
\hE \Norm{\tauboldh\nbfE}_{0,\F}^2
\le c_1 \hE^3 \Norm{\DoFbf^\F(\tauboldh)}_{\ell^2}^2
\qquad\qquad \forall \tauboldh \in \SigmaboldhE.
\end{equation}
We have the decomposition into polynomial fields on facets
\[
\tauboldh\nbfE{}_{|\F} =
\sum_{\vert \boldalpha \vert = 1}^{\p} \lambda^\F_{\boldalpha} \mboldalphaF
\qquad\qquad
\text{for given } \lambda^\F_{\boldalpha} \in \Rbb.
\]
Testing both sides by~$\mboldbetaF$,
scaling by the area of the facet~$\F$,
and using that the~$\tauboldh \nbfE$
are linear combinations
of scaled orthogonal polynomial fields, we deduce
\begin{equation} \label{relation-beta-facets}
\DoFbf^\F_{\boldbeta}(\tauboldh)
= \frac{1}{\vert \F \vert} \int_\F (\tauboldh\nbfE) \cdot \mboldbetaF
= \frac{1}{\vert \F \vert}
\sum_{\vert \boldalpha \vert = 0}^{\p} \lambda_{\boldalpha}^\F \int_\F \mboldalphaF \cdot \mboldbetaF .
\end{equation}
Collect the coefficients~$\lambda_{\boldalpha}^\F$ in a vector~$\boldlambda^\F$.
Using the arguments as in~\cite[Lemma~4.1]{Chen-Huang:2018},
there exists a positive constant $c_{CH}^A$,
which only depends on~$\rho$ in Section~\ref{section:introduction}
and~$\p$ in~\eqref{local-functions-strong}, such that
\[
\Norm{\tauboldh\nbfE}_{0,\F}^2
\le c_{CH}^A \hE^2 \Norm{\boldlambda^\F}_{\ell^2}^2.
\]
Further using~\eqref{relation-beta-facets}, we deduce~\eqref{equivalence-lower-bound-face}.

\paragraph*{The lower bound (part~2).}
We prove that there exists a positive constant~$c_2$ such that
\begin{equation} \label{equivalence-lower-bound-interior}
\hE^2 (\PiRMperp\divbf \sigmaboldh, \PiRMperp \divbf \tauboldh )_{0,\E}
\le c_2 \hE^3 (\DoFbfdiv(\sigmaboldh), \DoFbfdiv(\tauboldh))_{\ell^2}.
\end{equation}
We have the decomposition into polynomial fields,
which are orthogonal to rigid body motions,
\[
\PiRMperp \divbf \tauboldh
= \sum_{\vert \boldalpha \vert =0}^{\p}
\muperp_{\boldalpha} \mboldalphaperp
\qquad\qquad \text{for given } \muperp_{\boldalpha} \in \Rbb.
\]
Testing on both sides by~$\mboldbetaperp$
and using that the~$\mboldalphaperp$ are linear combinations
of scaled orthogonal polynomial fields, we deduce
\begin{equation} \label{relation-beta-interior}
\DoFbfdivboldbeta(\tauboldh)
= \frac{\hE}{\vert\E\vert} \int_\E \divbf \tauboldh \cdot \mboldbetaperp
= \sum_{\vert \boldalpha \vert = 1}^\p
\muperp_{\boldalpha} \frac{\hE}{\vert \E \vert} \int_\E \mboldalphaperp \cdot \mboldbetaperp.
\end{equation}
Collect the coefficients~$\muperp_{\boldalpha}$ in a vector~$\boldmu^\perp$.
Using the arguments for the 3D version of~\cite[Lemma~4.1]{Chen-Huang:2018},
there exists a positive constant~$c_{CH}^B$,
which only depends on~$\rho$ in Section~\ref{section:introduction}
and~$\p$ in~\eqref{local-functions-strong},
such that
\[
\Norm{\PiRMperp\divbf\tauboldh}_{0,\E}^2
\le c_{CH}^B \hE^{3} \Norm{\boldmu^\perp}_{\ell^2}^2
= c_{CH}^B \hE \Norm{ \hE \boldmu^\perp}_{\ell^2}^2.
\]
Further using~\eqref{relation-beta-interior}, we deduce~\eqref{equivalence-lower-bound-interior}.

Combining bounds~\eqref{equivalence-lower-bound-face} and~\eqref{equivalence-lower-bound-interior} yields the lower bound.

\paragraph*{The upper bound (part~1).}
We show, for all facets~$\F$ in~$\FcalE$, the existence
of a positive constant~$c_3$ such that
\begin{equation} \label{equivalence-upper-bound-face}
\hE^3 \Norm{\DoFbf^\F(\tauboldh)}_{\ell^2}^2
\le c_3 \hE \Norm{\tauboldh\nbfE}_{0,\F}^2
\qquad\qquad \forall \tauboldh \in \SigmaboldhE.
\end{equation}
Since~$\Norm{\mboldbetaF}_{L^{\infty}(\F)}=1$,
we have
\[
\DoFbf^\F_{\boldbeta}(\tauboldh)
=  \frac{1}{\vert\F\vert} \int_\F \tauboldh\nbfE \cdot \mboldbetaF
 \le \frac{1}{\vert \F \vert^\frac12} \Norm{\tauboldh\nbfE}_{0,\F}
 \le C_\rho \hE^{-1} \Norm{\tauboldh\nbfE}_{0,\F} ,
\]
for some positive constant $C_\rho$ only depending on~$\rho$
in Section~\ref{section:introduction}.

This gives
\[
\hE^3 \DoFbf^\F_{\boldbeta}(\tauboldh)^2
\le C_\rho^2 \hE \Norm{\tauboldh\nbfE}_{0,\F}^2.
\]
Summing over the correct multi-indices
gives~\eqref{equivalence-upper-bound-face}
for all facets~$\F$ of~$\E$.

\paragraph*{The upper bound (part~2).}
We prove the existence of a positive constant $c_4$ such that
\begin{equation} \label{equivalence-upper-bound-interior}
\hE^3 (\DoFbfdiv(\sigmaboldh), \DoFbfdiv(\tauboldh))_{\ell^2}
\le c_4 \hE^2 (\PiRMperp\divbf \sigmaboldh, \PiRMperp \divbf \tauboldh )_{0,\E}.
\end{equation}
Since~$\Norm{\mboldbetaperp}_{L^{\infty}(\E)}=1$, we have
\[
\DoFbfdivboldbeta(\tauboldh)
= \frac{\hE}{\vert \E \vert} \int_\E \divbf \tauboldh \mboldbetaperp
\le \frac{\hE}{\vert \E \vert^{\frac12}} \Norm{\PiRMperp \divbf\tauboldh}_{0,\E}
\le C_\rho \hE^{-\frac12} \Norm{\PiRMperp \divbf \tauboldh}_{0,\E},
\]
for some positive constant $C_\rho$ only depending on~$\rho$
in Section~\ref{section:introduction}.

This yields
\[
\hE^3 \DoFbfdivboldbeta(\tauboldh)^2
\le C_\rho^2 \hE^2 \Norm{\PiRMperp \divbf \tauboldh}_{0,\F}^2.
\]
Summing over the correct multi-indices
gives~\eqref{equivalence-upper-bound-interior}.
\end{proof}

%%%%%%

An immediate consequence of Propositions~\ref{proposition:stability-HR-general-order}
and~\ref{proposition:equivalence-two-stab} is the following result.
\begin{corollary} \label{corollary:stability-dofidofi-HR-general-order}
The bilinear form~$\StildeE(\cdot,\cdot)$ defined in~\eqref{stabilization:dofi-dofi}
satisfies the stability bounds~\eqref{stability-bounds}.
In fact, the result holds true also for tensors~$\tauboldh$
such that~$\PiTpE\tauboldh$ is different from the zero tensor.
\end{corollary}

\begin{remark} \label{remark:stab-bounds-zero-divergence}
In the virtual element discretization of~\eqref{weak-formulation},
displacements are vector, piecewise polynomials of degree~$\p$;
see \cite{Artioli-DeMiranda-Lovadina-Patruno:2018, Visinoni:2024}.
In particular, $\div\SigmaboldhE$ coincides with that space
and therefore the discrete stress solution
is also the solution to the reduced problem
on the discrete kernel of the mixed method.
Based on this, the stability bounds~\eqref{stability-bounds}
should be valid for divergence free discrete tensors
in $\tauboldh$ in $\SigmaboldhE \cap \ker (\PiTpE)$.
Reduced versions of the bilinear forms $\SE(\cdot,\cdot)$
and $\StildeE(\cdot,\cdot)$ may be employed:
in the former case, the divergence term vanishes;
in the latter, only facet degrees of freedom are employed
(this is in fact the original stabilization
proposed in~\cite{Artioli-DeMiranda-Lovadina-Patruno:2018, Visinoni:2024}).
On the theoretical level, we have the following advantages:
no polynomial inverse inequalities are needed
in the proof of the lower bound
in Proposition~\ref{proposition:stability-HR-general-order};
in the proof of Proposition~\ref{proposition:equivalence-two-stab},
the ``parts~2'' can be skipped.
\end{remark}
%%

%%%%%%%%%%%%%
\section{Interpolation estimates} \label{section:interpolation}
%%%%%%%%%%%%%
We derive interpolation estimates for functions in HR virtual element spaces
based on the stability bounds derived in Proposition~\ref{proposition:stability-HR-general-order}.
Related but different interpolation results can be found in
\cite[Proposition~5.3]{Artioli-DeMiranda-Lovadina-Patruno:2017},
\cite[Proposition 4.2]{Artioli-DeMiranda-Lovadina-Patruno:2018},
\cite[Proposition~4.2]{Dassi-Lovadina-Visinoni:2020},
and \cite[Proposition 4.4]{Visinoni:2024}.

Given an element~$\E$ and a sufficiently smooth
(in the sense of Remark~\ref{remark:regularity-stress} below) stress~$\sigmabold$,
we consider the unique discrete stress~$\sigmaboldI$ in~$\Sigmaboldh$
sharing the degrees of freedom of~$\sigmabold$.
More precisely, we define~$\sigmaboldI$
as the only function in~$\Sigmaboldh$
satisfying
\begin{equation} \label{interpolant}
\begin{split}
\int_\E \divbf(\sigmabold-\sigmaboldI) \cdot \qbfpperp=0
    & \qquad\qquad \forall \qbfpperp \in \RMperp(\E),\ \forall \E \in \taun, \\
\int_\F (\sigmabold-\sigmaboldI) \nbfE \cdot \qbfpF = 0
    & \qquad\qquad \forall \qbfpF \in \Pbf_\p(\F),\ \forall \F \in \FcalE,\ \forall \E \in \taun.
\end{split}
\end{equation}

%%%%%
\begin{remark} \label{remark:regularity-stress}
We need sufficient regularity for the stress~$\sigmabold$
in order to define its interpolant~$\sigmaboldI$ in the sense of~\eqref{interpolant}.
For instance, if~$\sigmabold$ belongs to
$\HScal(\divbf,\Omega) \cap \Hbb^{\frac12+\varepsilon}(\Omega)$, $\varepsilon>0$,
$\divbf\sigmabold$ belongs to~$\Lbf^2(\Omega)$
and $\sigmabold \nbfE {}_{|\F}$ belongs to~$\Lbf^2(\F)$
for all facets~$\F$ in the mesh,
whence the integrals in~\eqref{interpolant} are well defined.
Lower regularity can be demanded following, e.g.,
\cite[eq. (2.5.1)]{Boffi-Brezzi-Fortin:2013}
or \cite[Sect. 17.2]{Ern-Guermond:2021}:
it suffices that~$\sigmabold$ belongs to the space of symmetric stresses
with $\sigmabold$ in $\Lbb^{s}(\Omega)$, $s>2$,
and~$\divbf \sigmabold$ in $\Lbf^q(\Omega)$, $q>6/5$.
In this case, the integrals in~\eqref{interpolant}
make sense up to face-to-cell liftings.
\end{remark}
%%%%%

For~$\PiRMperp$ and~$\PizF$ as in~\eqref{RM-perp-L2-projection}
and~\eqref{L2-projector-face}, and sufficiently smooth stresses,
definition~\eqref{interpolant} gives
\begin{equation} \label{implication-interpolation}
\PiRMperp \divbf \sigmaboldI = \PiRMperp \divbf \sigmabold,
\qquad\qquad
(\sigmaboldI\nbfE){}_{|\F}
= \PizF (\sigmabold\nbfE){}_{|\F}
\quad \forall \F \in \FcalE.
\end{equation}
We have the following commutative property.
\begin{lemma}  \label{lemma:commutative-property}
Given $\sigmabold$ satisfying
one of the two assumptions in Remark~\ref{remark:regularity-stress}
and~$\sigmaboldI$ its interpolant in~$\Sigmaboldh$ as in~\eqref{interpolant},
we have the following identity:
\begin{equation} \label{commutative-property-divergence}
(\divbf \sigmaboldI)_{|\E}
= \PizE (\divbf \sigmabold)_{|\E}
\qquad\qquad \forall \E \in \taun.
\end{equation}
\end{lemma}
%%%%
\begin{proof}
The proof is based on the arguments
in Remark~\ref{remark:computability-divergence}.
The divergence of functions in~$\SigmaboldhE$ is a polynomial field.
Given~$\PiRM$ and~$\PiRMperp$ as in~\eqref{RM-L2-projection} and~\eqref{RM-perp-L2-projection},
we recall the orthogonal splitting
\begin{equation} \label{splitting-divergence}
(\divbf \sigmaboldI)_{|\E}
= \PiRM (\divbf \sigmaboldI)_{|\E}
    + \PiRMperp (\divbf \sigmaboldI)_{|\E} 
    \qquad\qquad \forall \E \in \taun.
\end{equation}
Using the first condition in~\eqref{interpolant}, we write
\small{\[
\int_\E \PiRMperp \divbf \sigmaboldI \cdot \qbfpperp
= \int_\E \divbf \sigmaboldI \cdot \qbfpperp
= \int_\E \divbf \sigmabold \cdot \qbfpperp
= \int_\E \PiRMperp \divbf \sigmabold \cdot \qbfpperp
\quad \forall \qbfpperp \in \RMperp(\E).
\]}\normalsize
This implies that
\[
\PiRMperp (\divbf \sigmaboldI)_{|\E}
= \PiRMperp (\divbf \sigmabold)_{|\E}
\qquad\qquad \forall \E \in \taun.
\]
On the other hand, integrating by parts,
using the second condition in~\eqref{interpolant},
and exploiting the fact that~$\nablaS(\qbfpRM)$ is zero
for all rigid body motions~$\qbfpRM$,
we also deduce
\footnote{The $L^2$ inner products on facets
should be replaced by suitable duality pairings
in case~$\sigmabold$ satisfies the second regularity
assumption in Remark~\ref{remark:regularity-stress}.}
\small{\[
\begin{split}
& \int_\E \PiRM \divbf \sigmaboldI \cdot \qbfpRM  
  = \int_\E \divbf \sigmaboldI \cdot \qbfpRM  
  = - \int_\E \sigmaboldI : \nablaS(\qbfpRM)
     + \sum_{\F \in \FcalE} \int_\F \sigmaboldI\nbfE \cdot \qbfpRM \\
& = - \int_\E \sigmabold : \nablaS(\qbfpRM)
     + \sum_{\F \in \FcalE} \int_\F \sigmabold\nbfE \cdot \qbfpRM
  = \int_\E \divbf \sigmabold \cdot \qbfpRM  
  = \int_\E \PiRM \divbf \sigmabold \cdot \qbfpRM .
\end{split}
\]}\normalsize
This implies that
\[
\PiRM (\divbf \sigmaboldI)_{|\E}
= \PiRM (\divbf \sigmabold)_{|\E}
\qquad\qquad\qquad\qquad \forall \E\in\taun.
\]
Inserting the displays above in~\eqref{splitting-divergence}
and recalling the orthogonal decomposition~\eqref{orthogonal-splitting},
we obtain
\[
(\divbf \sigmaboldI)_{|\E}
= \PiRM (\divbf \sigmabold)_{|\E} 
    + \PiRMperp (\divbf \sigmabold)_{|\E}
= \PizE (\divbf \sigmabold)_{|\E}
\qquad\qquad \forall \E \in \taun.
\]
\end{proof}
%%%%

Introduce the local spaces
\[
\SigmaboldtildeE
:= \{ \taubold \in \HScal(\divbf,\E) \mid 
            \exists\ \wbf \in \Hbf^1(\E) \text{ such that }
            \taubold = \Cbb \nablaS (\wbf) \}
\qquad \forall \E\in \taun.
\]
The following polynomial approximation result is valid;
see~\cite[Proposition~3.2]{Artioli-DeMiranda-Lovadina-Patruno:2018}
and \cite[Proposition~4.3]{Visinoni:2024}.

\begin{lemma} \label{lemma:approximation-PiT}
Let~$\sigmabold$ belong to $\SigmaboldtildeE \cap \Hbb^r(\E)$,
$r$ nonnegative, and $r$ be smaller than or equal to~$\p+1$.
Given~$\PiTpE$ as in~\eqref{definition:PiT},
the following bound holds true:
there exists a positive constant $C$ only depending on
the regularity parameter~$\rho$ in Section~\ref{section:introduction}
and~$\p$ in~\eqref{local-functions-strong}
such that   
\[
\Norm{\sigmabold - \PiTpE \sigmabold}_{0,\E}
\le C \hE^{r} \SemiNorm{\sigmabold}_{r,\E}.
\]
\end{lemma}
We are now in a position to prove interpolation estimates in HR virtual element spaces
in the $L^2$ norm and in the $L^2$ norm of the divergence.

\begin{theorem} \label{theorem:interpolation-HR}
If~$\sigmabold$ belongs to
$\SigmaboldtildeE \cap \Hbb^r(\E)$,
$r>1/2+\varepsilon$ for some~$\varepsilon>0$,
and~$r$ is smaller than or equal to~$\p+1$,
then there exists a positive constant~$c$
depending only on the shape of~$\E$
and~$\p$ in~\eqref{local-functions-strong} such that
\begin{equation} \label{interpolation-L2}
\Norm{\sigmabold - \sigmaboldI}_{0,\E}
\le c \hE^r \SemiNorm{\sigmabold}_{r,\E}
\qquad\qquad \forall \E \in \taun.
\end{equation}
If~$\sigmabold$ belongs to
$\SigmaboldtildeE \cap \Lbb^{s}(\Omega)$, $s>2$,
and~$\divbf \sigmabold$ in $\Hbf^t (\Omega)$,
$t\ge0$, and~$t$ is smaller than or equal to~$\p+1$,
then there exists a positive constant~$c$
depending only on the shape of~$\E$
and~$\p$ in~\eqref{local-functions-strong} such that
\begin{equation} \label{interpolation-L2-div}
\Norm{\div (\sigmabold-\sigmaboldI)}_{0,\E}
\le c \hE^t \SemiNorm{\divbf \sigmabold}_{t,\E}
\qquad\qquad \forall \E \in \taun.
\end{equation}
\end{theorem}
\begin{proof}
Bound~\eqref{interpolation-L2-div} follows from~\eqref{commutative-property-divergence}
and polynomial approximation properties.
In particular, the constant~$c$ is that of a best polynomial approximation result.
\medskip 

As for bound~\eqref{interpolation-L2},
we first observe that~\eqref{polynomials-in-space} entails~$\qbbpT = (\qbbpT)_I$
for any given~$\qbbpT$ in~$\TpE$.
Next, we apply the triangle inequality
and the lower bound as
in Proposition~\ref{proposition:stability-HR-general-order}
to get, for all~$\qbbpT$ in~$\TpE$,
\small{\[
\begin{split}
\Norm{\sigmabold-\sigmaboldI}_{0,\E}
\le \Norm{\sigmabold-\qbbpT}_{0,\E}
    + \Norm{(\sigmabold-\qbbpT)_I}_{0,\E}
\le \Norm{\sigmabold-\qbbpT}_{0,\E}
    + \alpha_*^{-1} \mu_{|\E} \SE((\sigmabold-\qbbpT)_I,(\sigmabold-\qbbpT)_I).
\end{split}
\]}\normalsize
Since the first term on the right-hand side yields optimal error estimates,
see Lemma~\ref{lemma:approximation-PiT},
with constant that only depends on
the regularity parameter~$\rho$ in Section~\ref{section:introduction}
and~$\p$ in~\eqref{local-functions-strong},
we only focus on the second one.

Putting aside the parameter~$\alpha_*^{-1}$,
and recalling~\eqref{stabilization:general-order} and~\eqref{implication-interpolation},
we write
\small{\[
\begin{split}
\mu_{|\E}
\SE((\sigmabold-\qbbpT)_I, (\sigmabold-\qbbpT)_I) 
& = \hE \sum_{\F \in \FcalE} \Norm{(\sigmabold-\qbbpT)_I \nbfE}^2_{0,\F}
+ \hE^2 \Norm{\PiRMperp \divbf (\sigmabold-\qbbpT)_I}^2_{0,\E} \\
& = \hE \sum_{\F \in \FcalE} \Norm{\PizF((\sigmabold-\qbbpT)\nbfE)}^2_{0,\F}
+ \hE^2 \Norm{\PiRMperp \divbf (\sigmabold-\qbbpT)}^2_{0,\E}.
\end{split}
\]}\normalsize
As for the first term on the right-hand side,
for each facet~$\F$,
we use the stability of the~$L^2$ projector on facets,
the trace \cite[Theorem 3.10]{Ern-Guermond:2021}
and the Poincar\'e inequality \cite[Section~3.3]{Ern-Guermond:2021},
the triangle inequality,
a polynomial inverse inequality~\cite{Verfurth:2013},
and get, for a suitable polynomial approximant~$\qbfunderp$ of~$\sigmabold$,
e.g., the tensor Bramble-Hilbert polynomial,
\begin{equation} \label{Bramble-Hilbert:triangle}
\begin{split}
& \hE \Norm{\PizF((\sigmabold-\qbbpT)\nbfE)}^2_{0,\F}
\le \hE \Norm{\sigmabold-\qbbpT}^2_{0,\F}
\lesssim \hE^{1+2\varepsilon} \SemiNorm{\sigmabold-\qbbpT}_{\frac12+\varepsilon,\E}^2\\
& \lesssim \hE^{1+2\varepsilon} 
        \Big(\SemiNorm{\sigmabold-\qbfunderp}_{\frac12+\varepsilon,\E}^2
          + \SemiNorm{\qbfunderp-\qbbpT}_{\frac12+\varepsilon,\E}^2 \Big)
    \lesssim \hE^{1+2\varepsilon} \SemiNorm{\sigmabold-\qbfunderp}_{\frac12+\varepsilon,\E}^2
          + \Norm{\qbfunderp-\qbbpT}_{0,\E}^2 \\
& \lesssim \hE^{1+2\varepsilon} \SemiNorm{\sigmabold-\qbfunderp}_{\frac12+\varepsilon,\E}^2
            + \Norm{\sigmabold-\qbbpT}_{0,\E}^2.
\end{split}
\end{equation}
The hidden constants depend on~$\rho$ in Section~\ref{section:introduction}
and~$\p$ in~\eqref{local-functions-strong}.
The above inequality gives
\small{\[
\begin{split}
\SE((\sigmabold-\qbbpT)_I, (\sigmabold-\qbbpT)_I) 
& \lesssim \hE^{1+2\varepsilon} \SemiNorm{\sigmabold-\qbfunderp}_{\frac12+\varepsilon,\E}^2
            + \Norm{\sigmabold-\qbbpT}_{0,\E}^2
  +\hE^2 \Norm{\PiRMperp \divbf (\sigmabold-\qbbpT)}^2_{0,\E}\\
& \le \hE^{1+2\varepsilon} \SemiNorm{\sigmabold-\qbfunderp}_{\frac12+\varepsilon,\E}^2
            + \Norm{\sigmabold-\qbbpT}_{0,\E}^2
  +\hE^2 \Norm{\divbf (\sigmabold-\qbbpT)}^2_{0,\E}.
\end{split}
\]}\normalsize
The first two terms on the right-hand side
converge optimally due to polynomial approximation estimates
and Lemma~\ref{lemma:approximation-PiT};
thence, we focus on the second one.
Proceeding as in~\eqref{Bramble-Hilbert:triangle}
leads us to
\[
\begin{split}
\Norm{\divbf (\sigmabold-\qbbpT)}_{0,\E}
& \le \Norm{\divbf (\sigmabold-\qbfunderp)}_{0,\E}
     + \Norm{\divbf (\qbfunderp-\qbbpT)}_{0,\E} \\
& \lesssim \Norm{\divbf (\sigmabold-\qbfunderp)}_{0,\E}
     + \hE^{-1 }\Norm{\qbfunderp-\qbbpT}_{0,\E} \\
& \le \Norm{\divbf (\sigmabold-\qbfunderp)}_{0,\E}
     + \hE^{-1} \Norm{\sigmabold- \qbfunderp}_{0,\E} 
     + \hE^{-1} \Norm{\sigmabold-\qbbpT}_{0,\E} .
\end{split}
\]
Polynomial approximation estimates and Lemma~\ref{lemma:approximation-PiT} yield the assertion.
\end{proof}

\begin{remark} \label{remark:Hdiv-convergence}
Bound~\eqref{interpolation-L2-div}
can be weakened
by taking the $\Lbf^q(\E)$ norm of $\divbf(\sigmabold-\sigmaboldI)$
and the (Banach) Sobolev $\mathbf W^{t,q}(\E)$ seminorm
of $\divbf\sigmabold$,
$t>0$, $q>6/5$, $t$ smaller than or equal to~$\p+1$,
on the left- and right-hand sides, respectively,
in case~$\sigmabold$ satisfies the second regularity
assumption in Remark~\ref{remark:regularity-stress}.
\end{remark}

%%%%%%%%%%%%%%%%%%%%%%%%%%%%%%%%%%%%%%%%%%%%%%%%%%%%%%%%%%%%%%%%%%%%%%%%%%
\section{A numerical investigation on the stability constants} \label{section:numerics}
%%%%%%%%%%%%%%%%%%%%%%%%%%%%%%%%%%%%%%%%%%%%%%%%%%%%%%%%%%%%%%%%%%%%%%%%%%

In this section, we discuss a practical approximation
of the coercivity and continuity constants
of the discrete bilinear forms in~\eqref{discrete:bf},
see Section~\ref{subsection:computation-stab-constants},
and assess numerically their behaviour
on sequences of elements
with increasing aspect ratio,
see Section~\ref{subsection:stability-h},
and with increasing degree of accuracy on fixed elements,
see Section~\ref{subsection:stability-p}.
We shall also be focusing on 2D elements.
The implementation is based on the C++ library Vem++~\cite{Dassi:2023}.

%%%%%%%%%%
\subsection{Computation of the coercivity and continuity constants} \label{subsection:computation-stab-constants}
%%%%%%%%%%
Given an element $\E$
and the dimension~$N^{\E}_{dof}$
of the space~$\SigmaboldhE$,
let $\big\{ \underline{\varphibold}_i \big\}_{i=1}^{N^{\E}_{dof}}$
be the basis of~$\SigmaboldhE$
dual to the degrees of freedom in~\eqref{face-moments}
and~\eqref{interior-moments}.
Define the symmetric matrices~$\Abf$ and~$\Bbf$
\begin{equation} \label{Abf-Bbf}
\Abf_{i,j} := \ahE(\underline{\varphibold}_j,\underline{\varphibold}_i),
\qquad\quad 
\Bbf_{i,j} := (\Dbb \underline{\varphibold}_j,\underline{\varphibold}_i)_{0,\E} 
\qquad\qquad \forall i,j = 1,\dots, N^{\E}_{dof}.
\end{equation}
In this section, we discuss a practical approximation
of the coercivity and continuity constants
of the discrete bilinear forms
$\ahE(\cdot, \cdot)$ in~\eqref{discrete:bf}. 
It can be readily seen that
these constants are the minimum and maximum
eigenvalues of the following
generalized eigenvalue problem:
find the eigenpair $(\gimel,\vbf)$ satisfying
\begin{equation} \label{generalizedEigenvalueProblem}
	\Abf \vbf = \gimel \Bbf \vbf.
\end{equation}
We recall that the bilinear forms~$\ahE(\cdot,\cdot)$
are computable via the degrees of freedom:
it suffices to compute the projector~$\PiTpE$
in~\eqref{definition:PiT}
and the stabilization.
Therefore, the computation of the matrix~$\Abf$
in~\eqref{Abf-Bbf} is rather simple.

Instead, the computation of the matrix~$\Bbf$
is less immediate: the virtual element basis tensors
are not available in closed form;
as such, they need to be approximated.
To this aim, we proceed in several steps.

\paragraph*{Step~1.}
We partition the set of local basis tensors
into basis tensors dual to facet~\eqref{face-moments}
and bulk~\eqref{interior-moments} moments:
\[
\left\{  \underline{\varphibold}^B \right\},
\qquad\qquad\qquad\qquad 
\left\{  \underline{\varphibold}^\perp \right\}.
\]
According to~\eqref{local-functions-strong},
a basis function $\underline{\varphibold}^B$
dual to~\eqref{face-moments}
and associated with a facet~$\F$ satisfies
\begin{equation} \label{face-basis-functions}
    \divbf \underline{\varphibold}^B =  \rbf \in \RM(\E) ,
    \qquad\qquad 
    \underline{\varphibold}^B \nbfE{}_{|\F} = 
    \qbf_{\p}^{\F} \in \Pbf_\p(\F),
    \qquad\qquad
    \underline{\varphibold}^B \nbfE{}_{|\partial\E\setminus\F} = 
    \zerobf  ;
\end{equation}
a bulk basis function $\underline{\varphibold}^\E$
dual to~\eqref{interior-moments} satisfies
\begin{equation} \label{bulk-basis-functions}
    \divbf \underline{\varphibold}^\perp =  \rbf^{\perp}_{\p} \in \RM^{\perp}(\E),
    \qquad\qquad\qquad\qquad
    \underline{\varphibold}^B \ \nbfE{}_{|\partial\E} = \zerobf.
\end{equation}
In~\eqref{face-basis-functions}, $\rbf$ is constructed
using only~$\qbf_{\p}^{\F}$;
see also Remark~\ref{remark:computability-divergence}.
In particular, compatibility conditions
are valid through the divergence theorem.

\paragraph*{Step~2.}
We define displacements
$\left\{\zbf_i\right\}_{i=1}^{N^{\E}_{dof}}$
such that we have
\begin{equation} \label{matrix-Bbf}
(\Dbb \underline{\varphibold}_j,\underline{\varphibold}_i)_{0,\E} 
= (\Cbb \nablaS \zbf_j,\nablaS \zbf_i)_{0,\E}
\qquad\qquad\qquad
\forall i,j=1,\dots,N^\E_{dof} . 
\end{equation} 
A concrete realization of such displacements
is given by the solutions
to ``face--type''
\begin{equation} \label{mixed-face-basis}
\begin{cases}
-\divbf (\Cbb \nablaS (\zbf_i^B)) 
\overset{\eqref{face-basis-functions}}{=} -\rbf_i   & \text{in } \E\\
\Cbb\nablaS(\zbf_i^B) \nbf
\overset{\eqref{face-basis-functions}}{=}\qbf_{\p,i}^{\F}    & \text{on } \F\\
\Cbb\nablaS(\zbf_i^B) \nbf
\overset{\eqref{face-basis-functions}}{=} \mathbf 0                & \text{on } \partial \E \setminus \F\\
\end{cases}
\end{equation}
with $\rbf_i$ satisfying a compatibility condition with
the Neumann boundary condition, and ``bulk--type''
\begin{equation} \label{mixed-bulk-basis}
\begin{cases}
-\divbf (\Cbb\nablaS (\zbf_i^\perp))
\overset{\eqref{bulk-basis-functions}}{=}
        -\rbf^{\perp}_{\p,i}               & \text{in } \E\\
   \Cbb\nablaS(\zbf_i^\perp) \nbf
   \overset{\eqref{bulk-basis-functions}}{=} \mathbf 0                   & \text{on } \partial \E,\\
\end{cases}
\end{equation}
elasticity problems.
The rigid body motion components of the solutions to
problems~\eqref{mixed-face-basis} and~\eqref{mixed-bulk-basis}
are fixed, e.g., by the conditions
\[
    \int_{\partial K} \zbf_i \cdot \rbf =\int_{\partial K}  \mathbf u\cdot \rbf
    \qquad\qquad\qquad\qquad
    \forall \rbf \in \RM(\E),
\]
where~$\ubf$ can be chosen arbitrarily.
In the following tests,
we pick $\ubf$ as an element in $\Pbf_\p(\E)$
such that its moments with respect to the elements
of the basis of scaled monomials with maximum
order~$\p$ are equal to~$1$.

\paragraph*{Step~3.}
We approximate the displacements,
and thence the matrix~$\Bbf$,
using a virtual element discretization
as in~\cite{BeiraoDaVeiga-Brezzi-Marini:2013}
on the element~$\E$.

%%%%%%%%%
\subsection{Stability constants on sequences of badly-shaped elements} \label{subsection:stability-h}
%%%%%%%%%
We assess numerically the behaviour of the
minimum (nonzero) and maximum eigenvalues
for the generalized eigenvalue
problem~\eqref{generalizedEigenvalueProblem}
for a fixed degree of accuracy
on sequences of elements with increasing aspect ratio
in two and three dimensions.

We consider both a compressible
($\lambda=\mu=1$)
and an incompressible materials
($\lambda=10^5$, $\mu=1$).
For the computation of the matrix~$\Bbf$
in~\eqref{matrix-Bbf}
we employ the VE scheme in~\cite{BeiraoDaVeiga-Brezzi-Marini:2013};
as for the latter case, we further
adopt a sub-integration of the divergence term,
which leads to a locking free scheme.

\paragraph*{First test case.}
We consider a sequence of elements
as in Figure~\ref{fig:hourglass2D}.
The initial element is a nonconvex hexagon
with hourglass shape; the distance between
the two re-entrant vertices is~$0.5$,
while that of the other couples of vertices is~$1$.
The other elements are constructed by halving
the distance between the two re-entrant
corners at each step.
We consider compressible materials.
\begin{figure}[ht]
	\centering
	\includegraphics[width=\sizeMesh\textwidth]{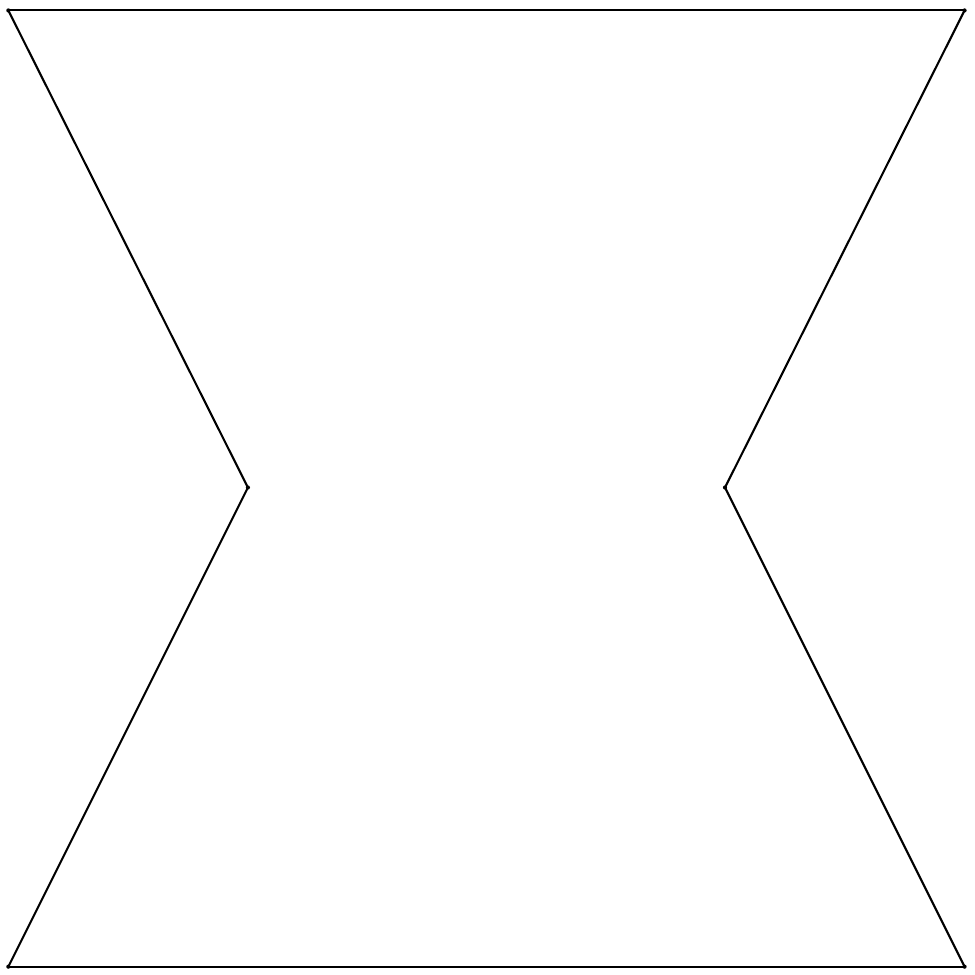}
	\includegraphics[width=\sizeMesh\textwidth]{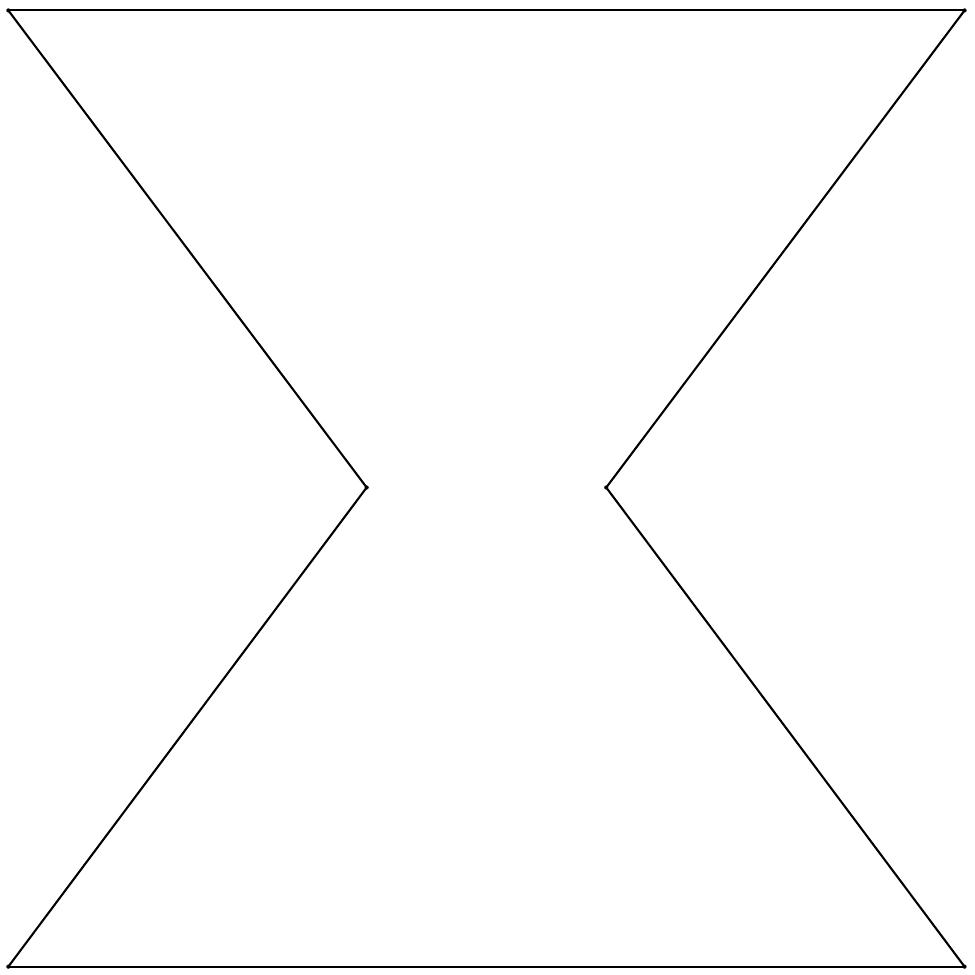}
	\includegraphics[width=\sizeMesh\textwidth]{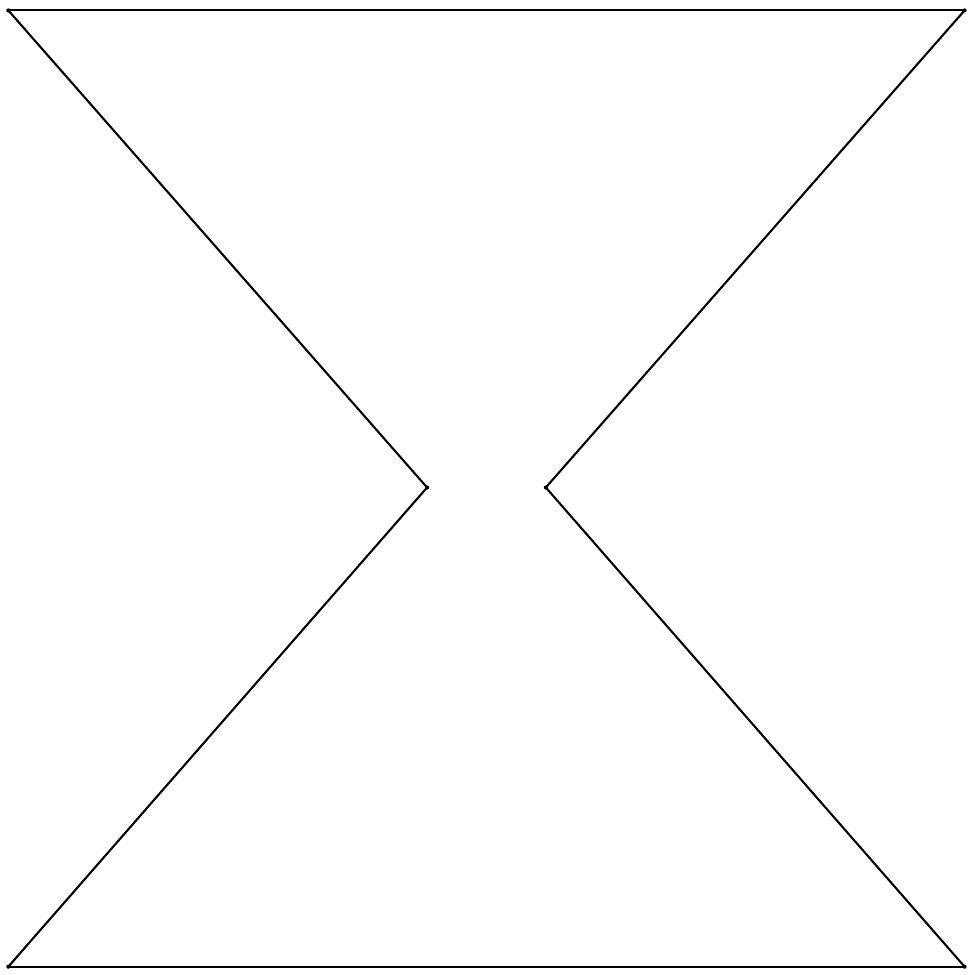}
	\includegraphics[width=\sizeMesh\textwidth]{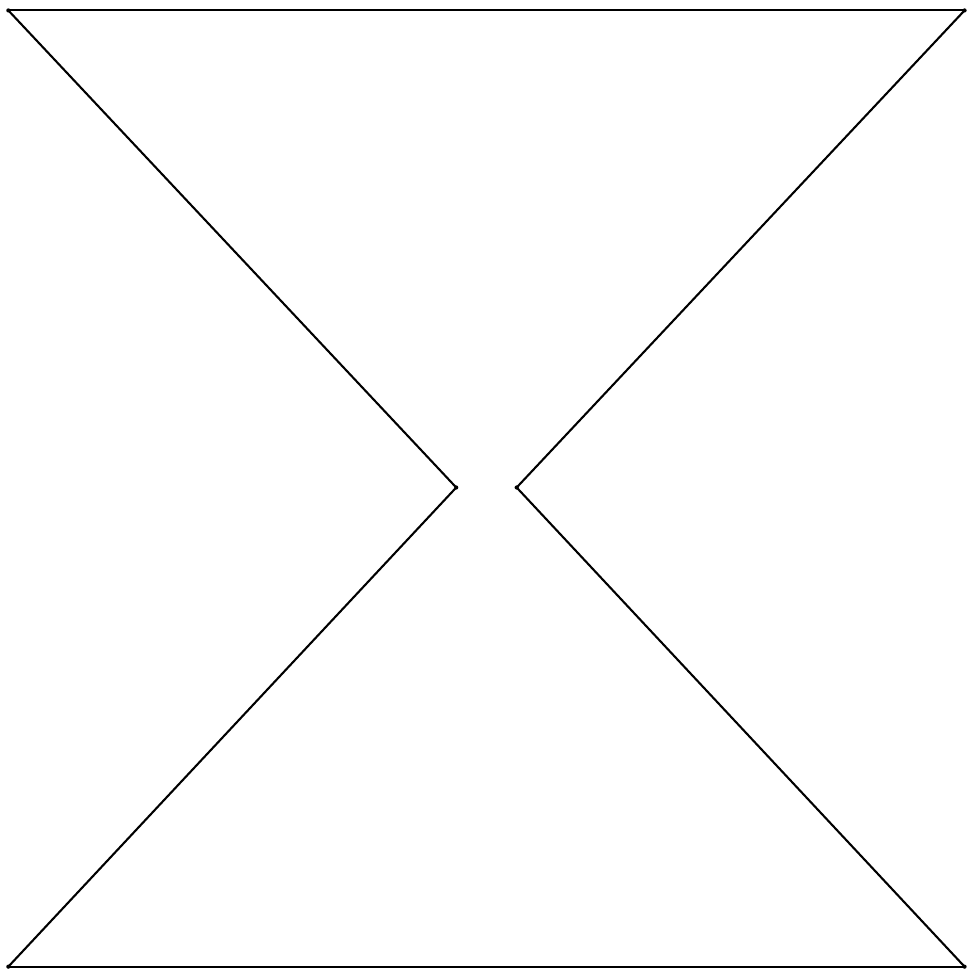}
	\includegraphics[width=\sizeMesh\textwidth]{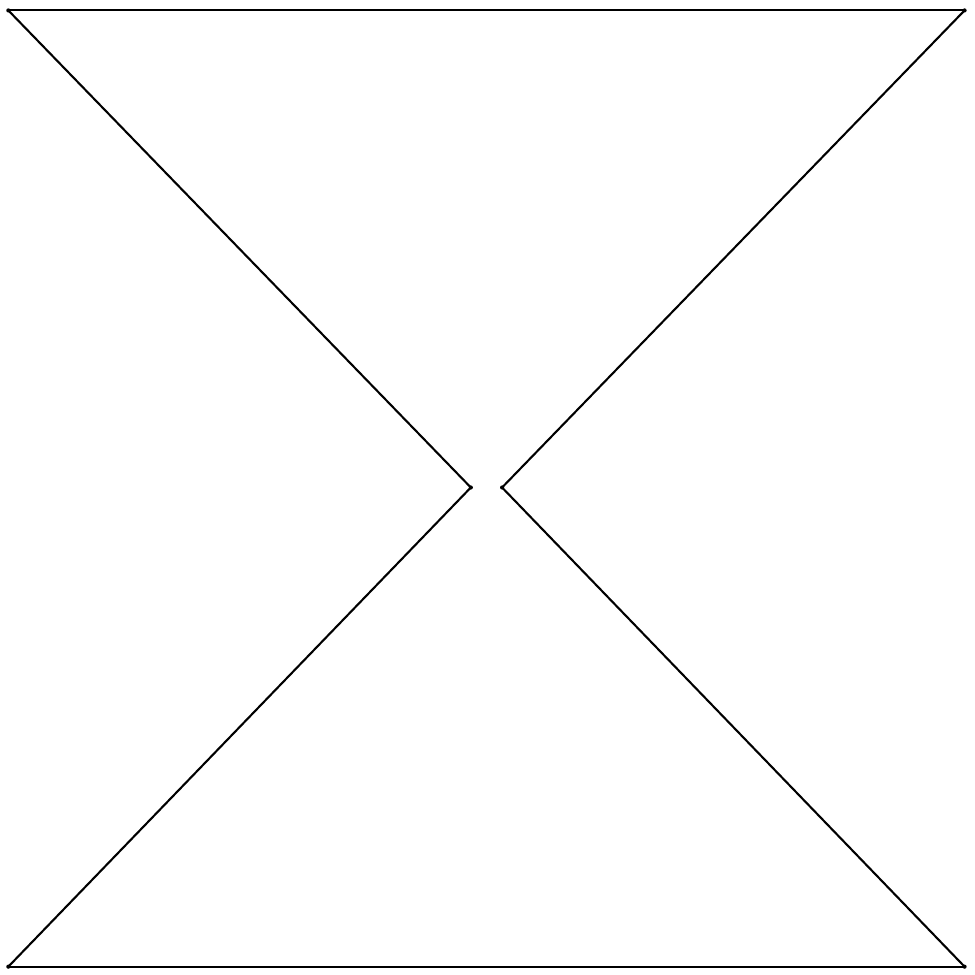}
	\includegraphics[width=\sizeMesh\textwidth]{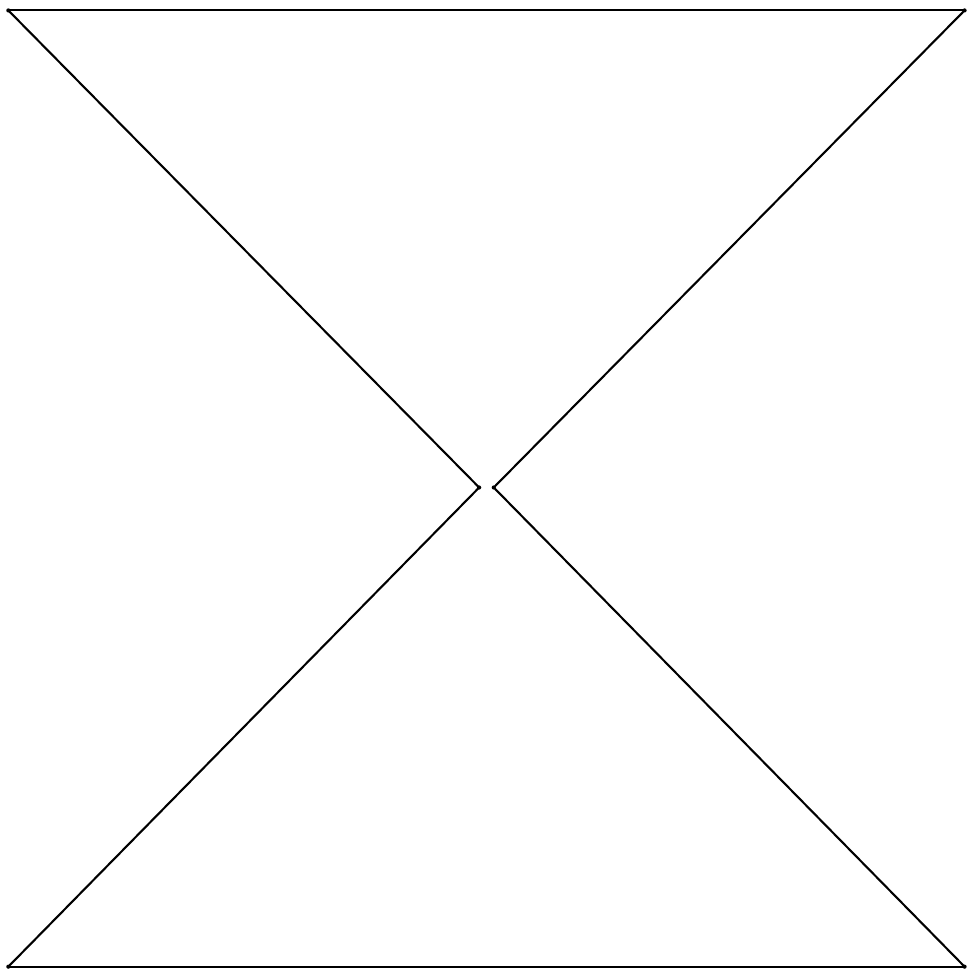}
	\caption{A sequence of 2D hourglass shaped elements
    for a compressible material.}
	\label{fig:hourglass2D}
\end{figure}

%%%
\paragraph*{Second test case.}
We consider a sequence of elements
as in Figure~\ref{fig:hourglass3D}.
The initial element is a nonconvex decahedron
with hourglass shape; the area of the upper and lower
facets is~$1$, while that of the minimal square section
is~$1/4$.
The other elements are constructed by halving
the size of the edges of the minimal square section.
We consider compressible materials.
\begin{figure}[ht]
	\centering
	\includegraphics[width=\sizeMesh\textwidth]{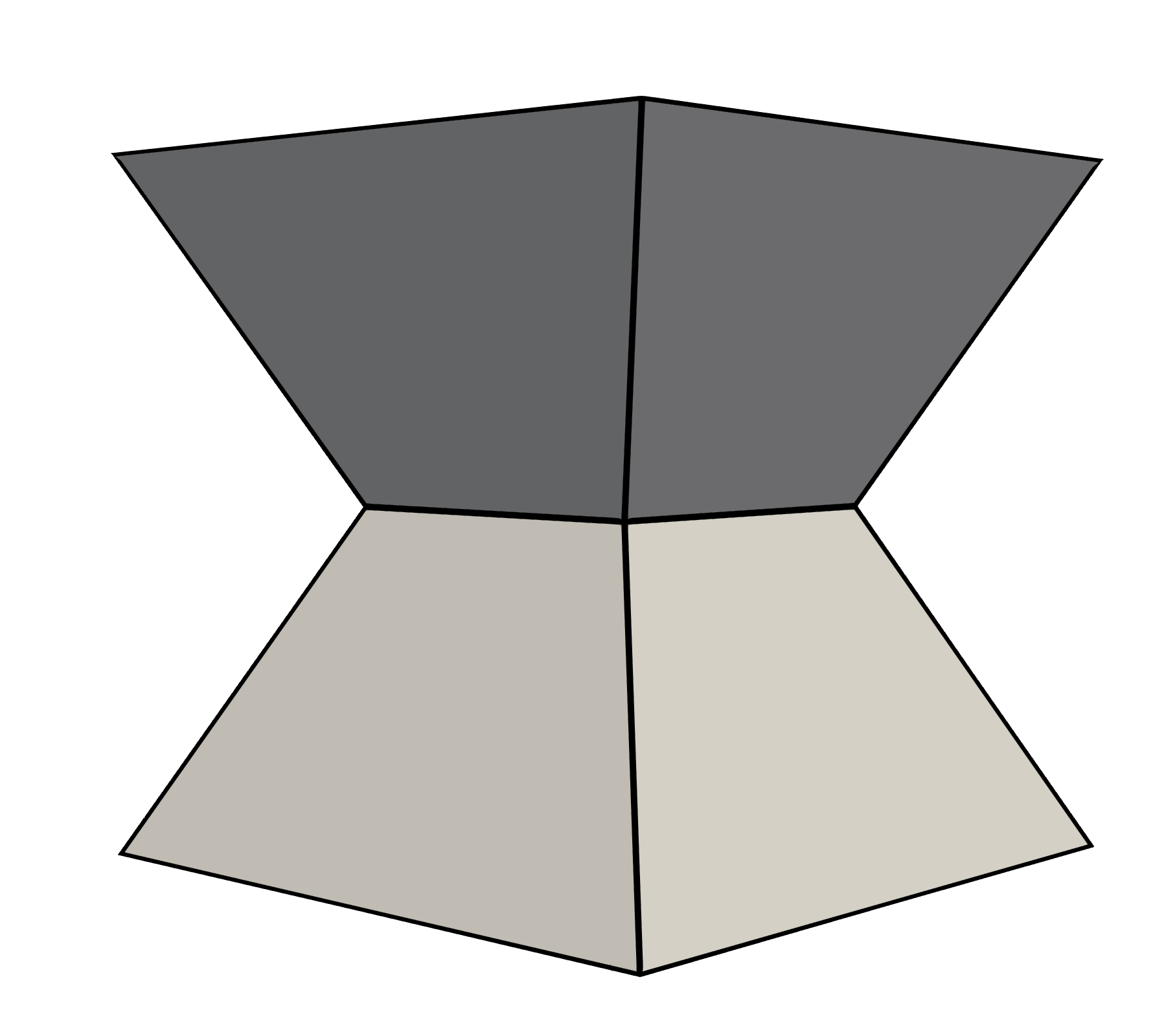}
	\includegraphics[width=\sizeMesh\textwidth]{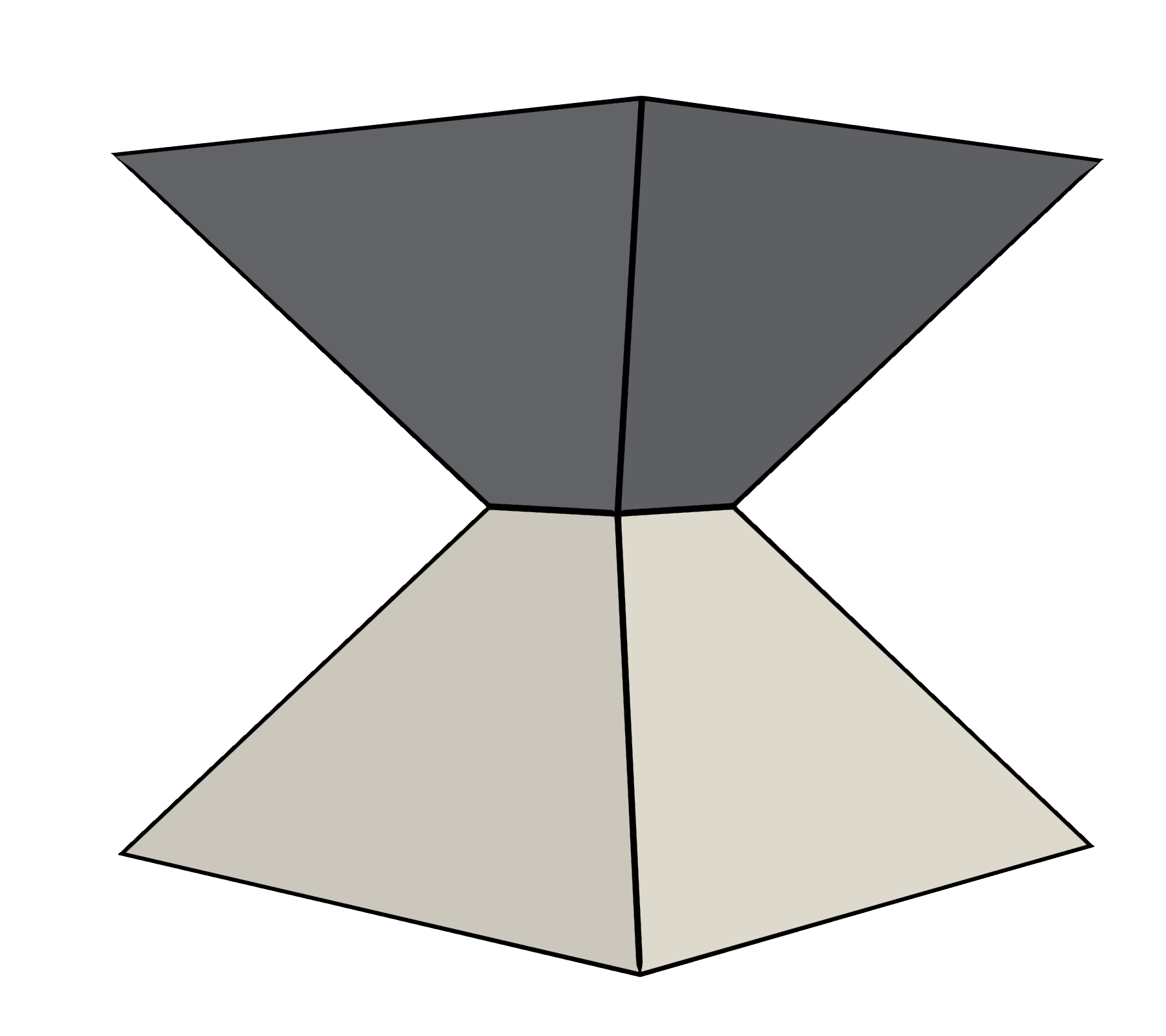}
	\includegraphics[width=\sizeMesh\textwidth]{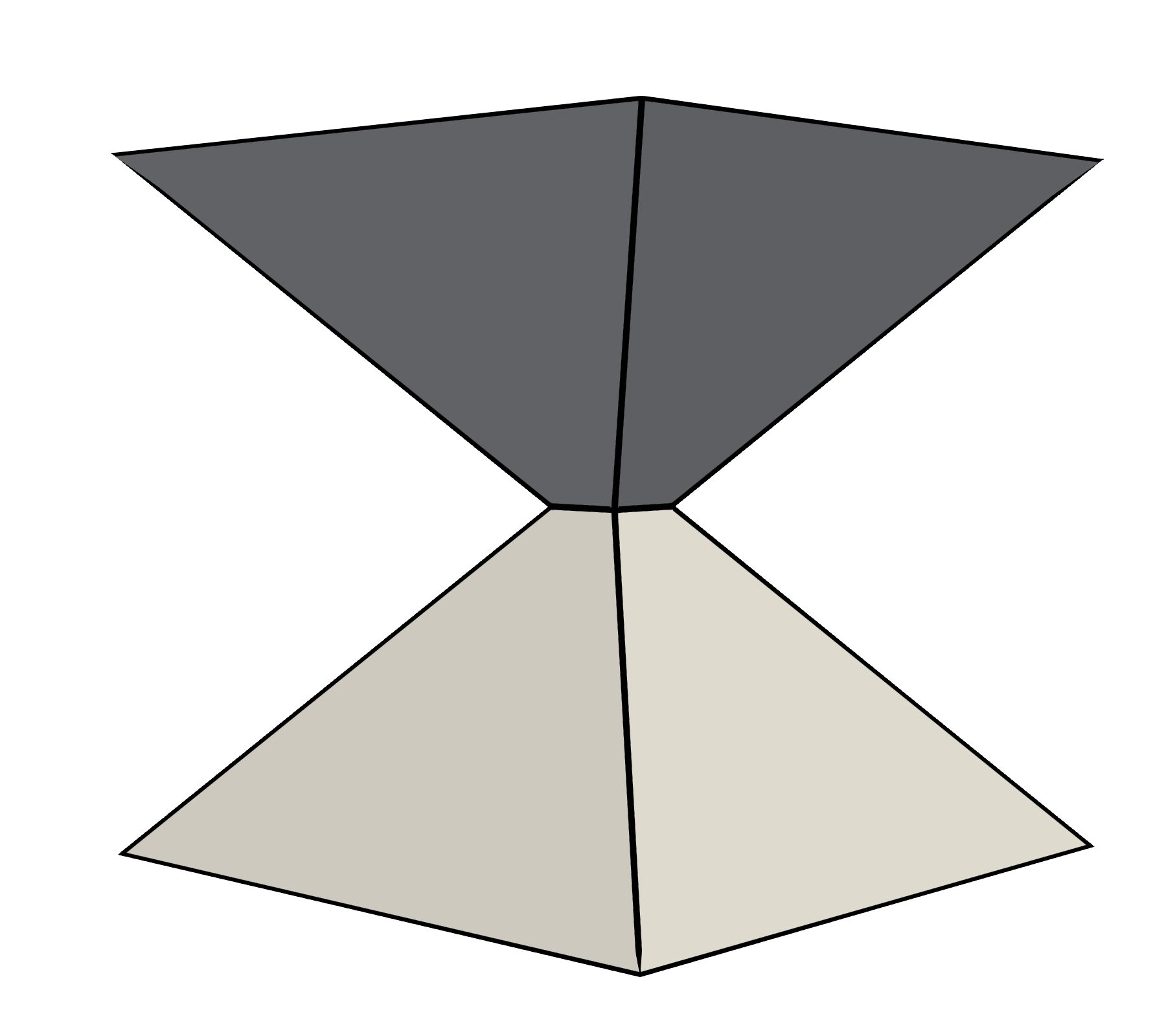}
	\includegraphics[width=\sizeMesh\textwidth]{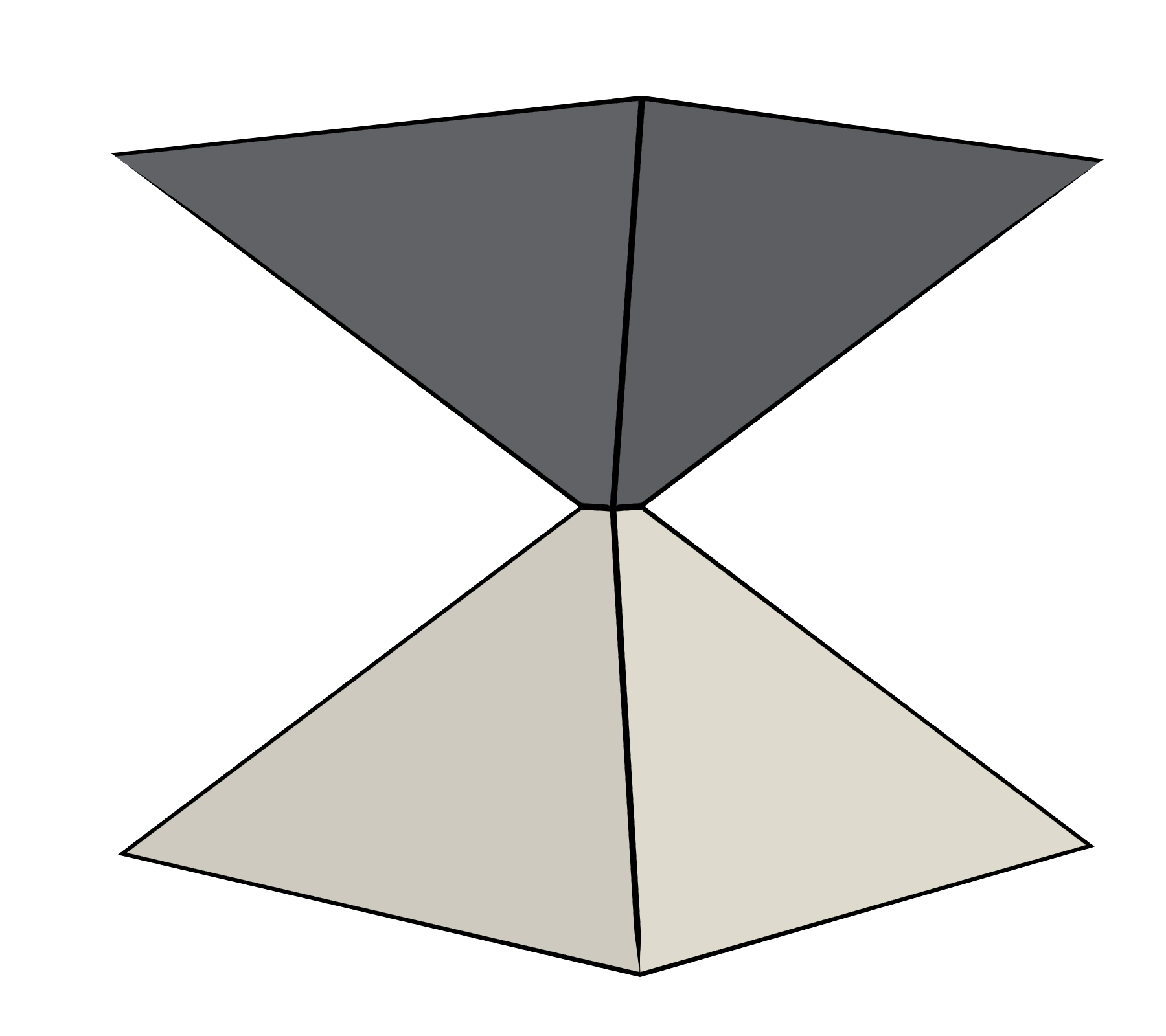}
	\includegraphics[width=\sizeMesh\textwidth]{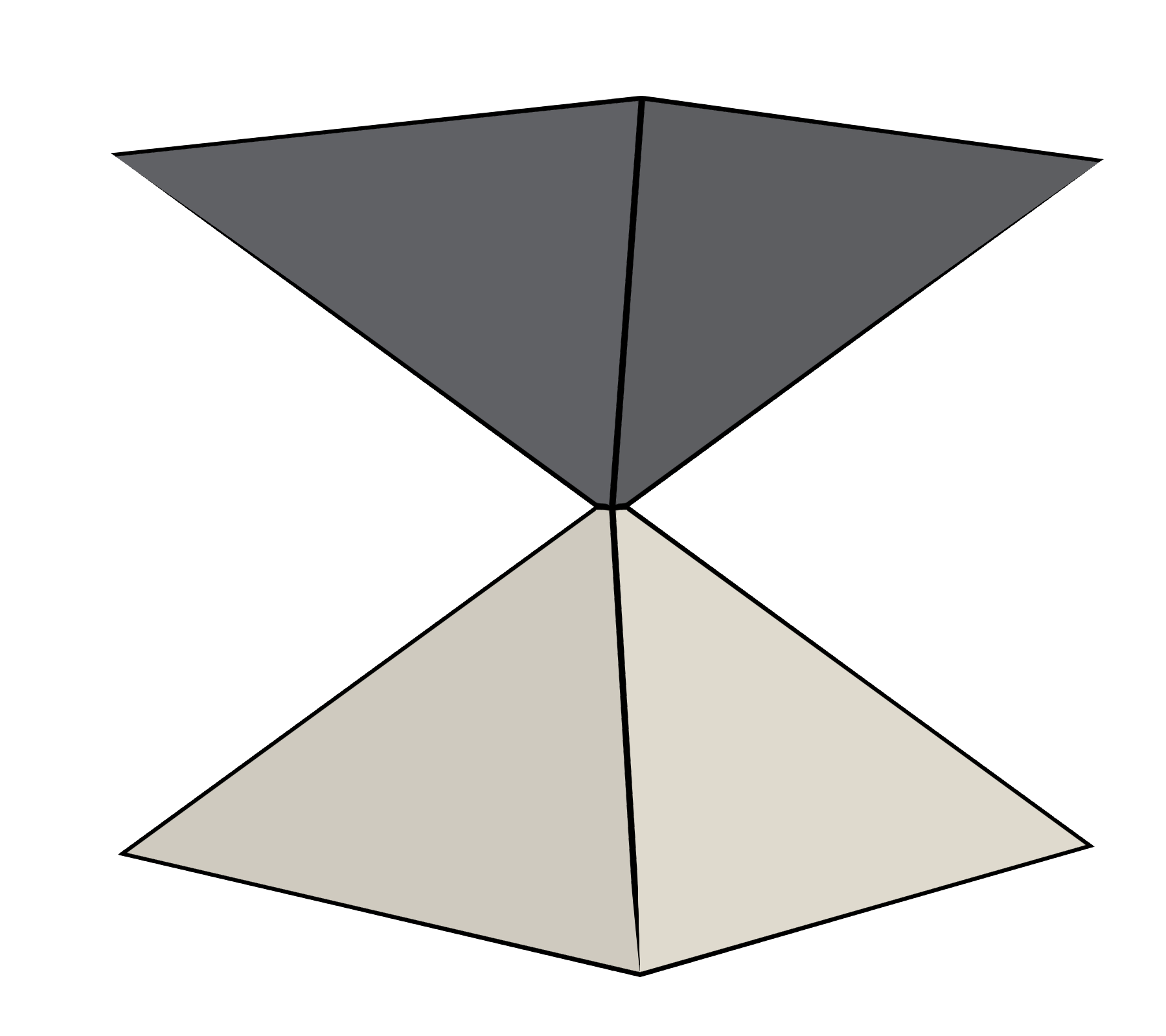}
	\includegraphics[width=\sizeMesh\textwidth]{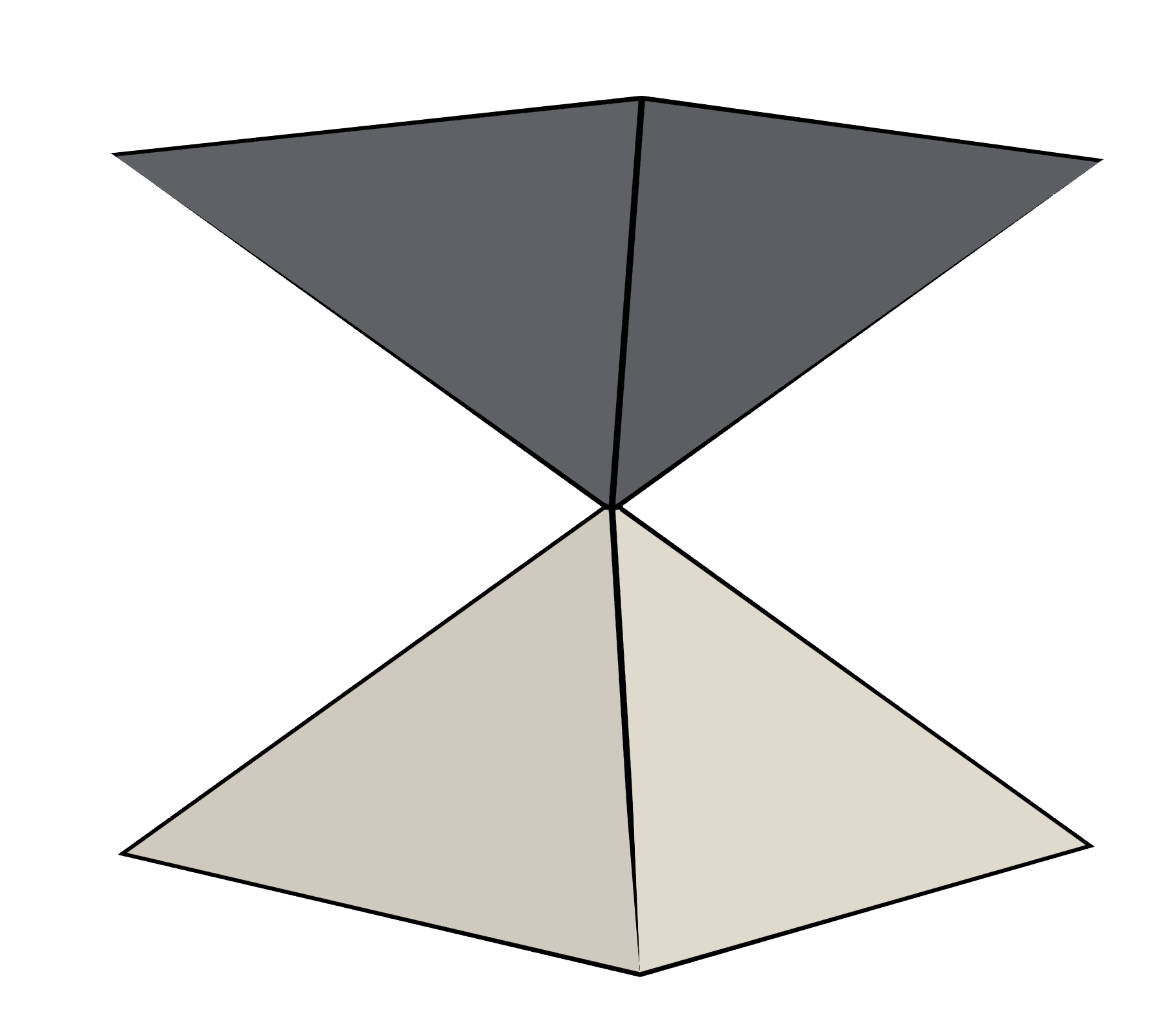}
	\caption{A sequence of 3D hourglass shaped elements
    for a compressible material.}
	\label{fig:hourglass3D}
\end{figure}

\paragraph*{Third test case.}
We consider a sequence of elements
as in Figure~\ref{fig:compressedTrapezoid}.
The initial element is an isosceles trapezoid,
where the length of the top edge is~$1/2$
and that of the bottom edge is~$1$.
The other elements are obtained by halving
the distance between the bottom and top edges.
We consider incompressible materials.
\begin{figure}[ht]
	\centering
	\includegraphics[width=\sizeMesh\textwidth]{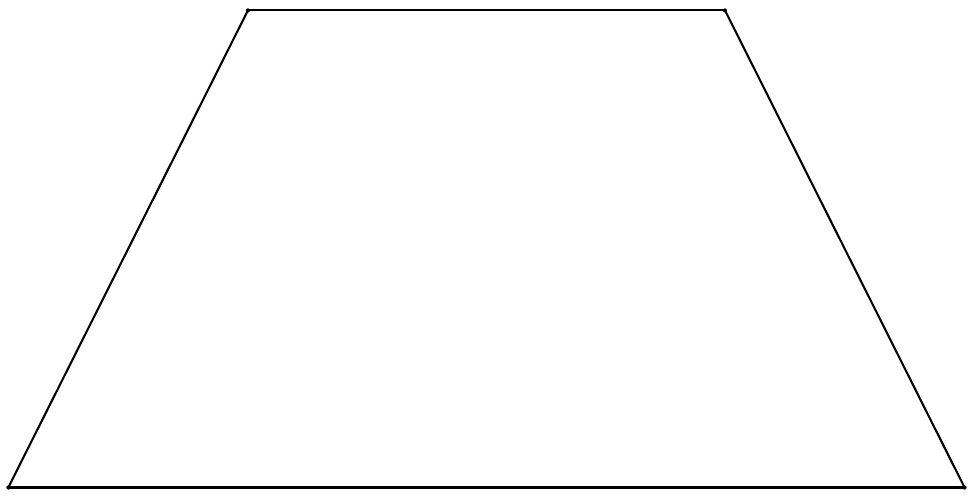}
	\includegraphics[width=\sizeMesh\textwidth]{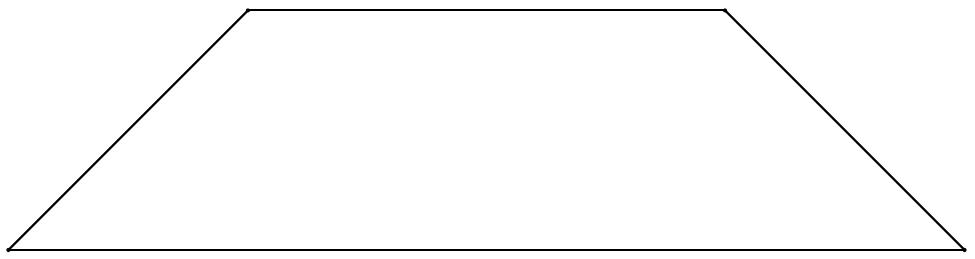}
	\includegraphics[width=\sizeMesh\textwidth]{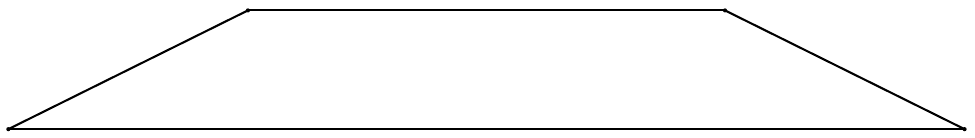}
	\includegraphics[width=\sizeMesh\textwidth]{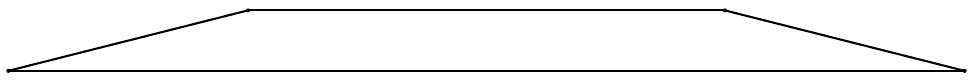}
	\includegraphics[width=\sizeMesh\textwidth]{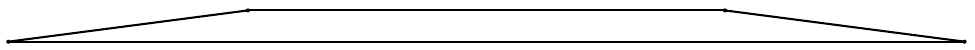}
	\includegraphics[width=\sizeMesh\textwidth]{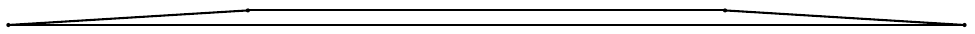}    
\caption{A sequence of 2D trapezoidal elements
    for an incompressible material.}
\label{fig:compressedTrapezoid}
\end{figure}

%%%
\paragraph*{Numerical results: minimum and maximum generalized eigenvalues.}
%%%
In Table~\ref{table:eigenvalues_badlyShapedElements},
we display the minimum (nonzero) and maximum generalized eigenvalues of~\eqref{generalizedEigenvalueProblem}.
For the 2D and 3D sequences of
hourglass shaped elements
in Figures~\ref{fig:hourglass2D}and~\ref{fig:hourglass3D},
we consider degree of accuracy $p=1$;
for the 2D sequences of trapezoidal elements
in Figure~\ref{fig:compressedTrapezoid}, we set $p=2$.
We provide results for the projection-based \eqref{stabilization:general-order}
and the dofi-dofi~\eqref{stabilization:dofi-dofi}
stabilizations $\SE(\cdot,\cdot)$ and $\StildeE(\cdot,\cdot)$;
in what follows, for the latter stabilization,
we denote the matrix~$\Abf$ by~$\widetilde{\Abf}$.

\begin{table}[!ht] 
\begin{center}
\begin{tabular}{ccc|cc}
\multicolumn{1}{c}{} &
\multicolumn{2}{c}{$\SE(\cdot,\cdot)$} &
\multicolumn{2}{c}{$\StildeE(\cdot,\cdot)$}\\
\hline
\multicolumn{1}{c}{$p=1$} &\multicolumn{1}{c}{$\gimel_{min}$} & \multicolumn{1}{c}{$\gimel_{max}$} &\multicolumn{1}{c}{$\gimel_{min}$} & \multicolumn{1}{c}{$\gimel_{max}$}\\
\hline
\multicolumn{1}{c|}{Fig.~\ref{fig:hourglass2D}} &{1.0000e+00} &{2.7301e+02} &{3.4204e-01} &{2.7849e+02} \\
\multicolumn{1}{c|}{} &{1.0000e+00} &{2.3326e+02} &{2.2719e-01}&{2.3492e+02}\\
\multicolumn{1}{c|}{} &{6.4300e-01} &{2.0826e+02} &{1.7951e-01}&{2.0788e+02}\\
\multicolumn{1}{c|}{} &{2.0247e-01} &{1.9905e+02} &{1.4982e-01} &{1.9722e+02}\\
\multicolumn{1}{c|}{} &{7.8824e-02} &{1.9528e+02} &{6.7338e-02} &{1.9277e+02} \\
\multicolumn{1}{c|}{} &{2.7496e-02} &{1.9338e+02}  &{2.3740e-02} &{1.9065e+02} \\
\hline
\multicolumn{1}{c}{$p=1$} &\multicolumn{1}{c}{$\gimel_{min}$} & \multicolumn{1}{c}{$\gimel_{max}$} &\multicolumn{1}{c}{$\gimel_{min}$} & \multicolumn{1}{c}{$\gimel_{max}$}\\
\hline
\multicolumn{1}{c|}{Fig.~\ref{fig:hourglass3D}}&{1.2260e-01} &{7.5879e+04}  &{1.2497e-01} &{5.6961e+04} \\
\multicolumn{1}{c|}{} &{1.3380e-02} &{8.8178e+04}  &{1.3875e-02} &{8.4101e+04} \\
\multicolumn{1}{c|}{} &{1.9540e-03} &{9.6572e+04} &{2.0313e-03} &{9.9241e+04} \\
\multicolumn{1}{c|}{} &{3.6127e-04} &{1.0486e+05}  &{3.7564e-04} &{1.1105e+05} \\
\multicolumn{1}{c|}{} &{8.1663e-05} &{1.1083e+05}  &{8.4897e-05} &{1.1890e+05} \\
\multicolumn{1}{c|}{} &{1.5835e-05} &{1.1373e+05}  &{1.6496e-05} &{1.2314e+05} \\
\hline
\multicolumn{1}{c}{$p=2$} &\multicolumn{1}{c}{$\gimel_{min}$} & \multicolumn{1}{c}{$\gimel_{max}$} &\multicolumn{1}{c}{$\gimel_{min}$} & \multicolumn{1}{c}{$\gimel_{max}$}\\
\hline
\multicolumn{1}{c|}{Fig.~\ref{fig:compressedTrapezoid}} &{1.9996e-05} &{6.5936e+02}  &{1.9996e-05} &{6.7161e+02} \\
\multicolumn{1}{c|}{} &{1.9980e-05} &{1.0625e+03}  &{1.9980e-05} &{1.0757e+03} \\
\multicolumn{1}{c|}{} &{1.9758e-05} &{1.4710e+03}  &{1.9758e-05} &{1.4801e+03} \\
\multicolumn{1}{c|}{} &{1.1262e-05} &{3.1641e+03}  &{1.1262e-05} &{3.1704e+03} \\
\multicolumn{1}{c|}{} &{2.0184e-10} &{2.3531e+03}  &{2.0229e-10} &{2.3541e+03} \\
\multicolumn{1}{c|}{} &{1.5359e-10} &{2.7322e+03}  &{1.4313e-10} &{2.7215e+03} \\
	\end{tabular}
\caption{Minimum and maximum eigenvalues of~\eqref{generalizedEigenvalueProblem}
on sequences of badly-shaped elements.} \label{table:eigenvalues_badlyShapedElements}
\end{center}
\end{table}

From the results in Table~\ref{table:eigenvalues_badlyShapedElements},
for the case of sequences of elements
as in Figure~\ref{fig:hourglass2D}, we observe
that the minimum eigenvalues decrease moderately,
while the maximum eigenvalues essentially do not increase;
for the case of sequences of elements
as in Figure~\ref{fig:hourglass3D},
a less moderate growth of the maximum eigenvalues
is displayed,
while the minimum eigenvalues lose
one order of magnitude at each halving step;
for the case of sequences of elements
as in Figure~\ref{fig:compressedTrapezoid},
the maximum eigenvalues grow slowly
and the minimum eigenvalues are essentially constant,
with the exception of the last two elements
of the sequence;
we shall motivate this behaviour in the next paragraph.

%%%
\paragraph*{Numerical results: condition numbers.}
%%%
Under the same choices as in the previous paragraph,
in Table~\ref{table:conditionNumber_badlyShapedElements},
we assess the behaviour of the condition numbers
of the matrices~$\Abf$, $\widetilde\Abf$, and~$\Bbf$
in problem~\eqref{generalizedEigenvalueProblem};
we recall that the matrices~$\Abf$ and~$\widetilde\Abf$
are computed based on the projection-based
\eqref{stabilization:general-order}
and dofi-dofi \eqref{stabilization:dofi-dofi}
stabilizations.

\begin{table}[!ht]
	\begin{center}
		\begin{tabular}{cccccccc}
\multicolumn{1}{c|}{Fig.~\ref{fig:hourglass2D}}&{$\Bbf$} &{5.8348e+03} &{2.2734e+04} &{7.7885e+04} &{2.4295e+05} &{6.2642e+05} &{1.7875e+06} \\
\multicolumn{1}{c|}{} & {$\Abf$} &{4.3858e+03}  &{8.7185e+03}  &{1.1033e+04}  &{1.2342e+04}  &{1.3042e+04}  &{1.3395e+04} \\
\multicolumn{1}{c|}{}&{$\widetilde\Abf$}  &{3.4318e+03}&{6.3807e+03}&{7.4790e+03}&{7.9627e+03}&{8.1894e+03}&{8.2922e+03}\\
\hline
\multicolumn{1}{c|}{Fig.~\ref{fig:hourglass3D}}&{$\Bbf$} &{3.4677e+04} &{9.3584e+04} &{2.0927e+05} &{5.4514e+05} &{2.3226e+06} &{2.5258e+07}\\ 
\multicolumn{1}{c|}{}&{$\Abf$}&  {9.9945e+03} &{1.8980e+04} &{2.6203e+04} &{3.0440e+04} &{3.2767e+04} &{3.3994e+04}\\
\multicolumn{1}{c|}{}&{$\widetilde\Abf$} &{6.7282e+03}&{1.4792e+04}&{2.1655e+04}&{2.5956e+04}&{2.8424e+04}&{2.9758e+04}\\
\hline
\multicolumn{1}{c|}{Fig.~\ref{fig:compressedTrapezoid}}&{$\Bbf$} & {2.2076e+05} & {2.1286e+06} & {3.5757e+07} & {1.1220e+09} & {4.4996e+09} & {2.3369e+10}\\
\multicolumn{1}{c|}{}&{$\Abf$}& {2.4149e+10} & {1.1371e+11} & {1.2118e+12} & {2.2248e+13} & {5.6501e+14} & {1.6530e+16} \\
\multicolumn{1}{c|}{}&{$\widetilde\Abf$} & {1.3668e+10} & {6.9767e+10} & {9.0468e+11} & {1.9766e+13} & {5.4511e+14} & {1.6521e+16}
		\end{tabular}
\caption{Condition numbers of~$\Abf$, $\widetilde{\Abf}$,
and~$\Bbf$ in~\eqref{generalizedEigenvalueProblem}
on sequences of badly-shaped elements.} \label{table:conditionNumber_badlyShapedElements}
\end{center}
\end{table}
%%%%

From the results in Table~\ref{table:conditionNumber_badlyShapedElements},
we deduce some facts:
the condition numbers of $\Abf$ and~$\widetilde{\Abf}$
are essentially the same,
the latter being always slightly smaller
than the former;
there is no clear indication on whether
the condition numbers of~$\Bbf$
are larger or smaller than those of~$\Abf$
and~$\widetilde{\Abf}$;
the condition numbers for the last two
trapezoidal elements are rather high,
which suggests a reason why the corresponding
minimum eigenvalues from the previous
paragraph looked unreliable.

%%%%%%%%%%%%%%%%%
\subsection{Stability constants increasing the degree of accuracy} \label{subsection:stability-p}
%%%%%%%%%%%%%%%%%
We assess numerically the behaviour of the
minimum (nonzero) and maximum eigenvalues
for the generalized eigenvalue
problem~\eqref{generalizedEigenvalueProblem}
for increasing degree of accuracy
on a fixed triangular element
of vertices $(0,0)$, $(1,0)$, and $(0,1)$
in a compressible material;
see Table~\ref{table:p-version}.
We provide results for the projection-based \eqref{stabilization:general-order}
and the dofi-dofi~\eqref{stabilization:dofi-dofi}
stabilizations $\SE(\cdot,\cdot)$ and $\StildeE(\cdot,\cdot)$.

%%%
\begin{table}
\begin{center}
\begin{tabular}{ccc|cc}
\multicolumn{1}{c}{} &
\multicolumn{2}{c}{$\SE(\cdot,\cdot)$} &
\multicolumn{2}{c}{$\StildeE(\cdot,\cdot)$}\\
\hline
\multicolumn{1}{c}{$p$} &\multicolumn{1}{c}{$\gimel_{min}$} & \multicolumn{1}{c}{$\gimel_{max}$} &\multicolumn{1}{c}{$\gimel_{min}$} & \multicolumn{1}{c}{$\gimel_{max}$}\\
\hline
\multicolumn{1}{c|}{1} &{1.0000e+00} &{1.7600e+02}  &{2.6323e-01} &{1.7398e+02}\\ 
\multicolumn{1}{c|}{2} &{9.9802e-01} &{5.9800e+02}  &{4.3155e-01} &{5.7797e+02}\\ 
\multicolumn{1}{c|}{3} &{9.9640e-01} &{1.2708e+03}  &{9.3328e-01} &{1.2491e+03}\\ 
\multicolumn{1}{c|}{4} &{9.9057e-01} &{2.3578e+03}  &{9.9057e-01} &{2.3313e+03}\\ 
\multicolumn{1}{c|}{5} &{9.6450e-01} &{4.1013e+03}  &{9.6377e-01} &{4.0739e+03}\\ 
\multicolumn{1}{c|}{6} &{9.3644e-01} &{6.8462e+03}  &{9.3645e-01} &{6.8147e+03}\\ 
\end{tabular}
\caption{Minimum and maximum eigenvalues of~\eqref{generalizedEigenvalueProblem}
for variable degrees of accuracy.}
\label{table:p-version}
\end{center}
\end{table}
%%%

From the results in Table~\ref{table:p-version},
it appears that the minimum and maximum eigenvalues
behave rather robustly with respect to the degree
of accuracy.

%%%%%%%%%%%%%%%%%%%%%%%%%%%%%%%%%%%%%%%%%%%%%%%%%%%%%%%%%%%%%%%%%%%%%%%%%%
\section{Conclusions} \label{section:conclusions}
%%%%%%%%%%%%%%%%%%%%%%%%%%%%%%%%%%%%%%%%%%%%%%%%%%%%%%%%%%%%%%%%%%%%%%%%%%
For shape-regular sequences of polytopic meshes,
we derived rigorously interpolation and stability
estimates for HR-type virtual elements
with fixed degree of accuracy in three dimensions
(the two dimensional case is a low hanging fruit
of the analysis given in this work).
Essential tools in the analysis are integrations by parts,
polynomial inverse estimates over suitable
sub-tessellations of the polytopic elements,
direct estimates,
polynomial approximation results,
and the proof of the well-posedness
(with a priori estimates on the data involving
constants that are explicit
with respect to the shape of the domain)
of mixed formulations
of linear elasticity problems.
Besides, we investigated on the numerical level
the behaviour of the stabilization
on sequences of badly-shaped elements
and increasing degree of accuracy
employing two different stabilizations:
standard assumptions on the geometry
are not enough to guarantee the robustness of the method.
Future investigation will cope with the construction
of ad-hoc stabilizations leading to stable
bilinear forms on degenerating geometries
and increasing degree of accuracy.

%%%%%%%%%%%%%%
\paragraph*{Acknowledgments.}
MB and LM have been partially funded by the
European Union (ERC, NEMESIS, project number 101115663);
views and opinions expressed are however those
of the authors only and do not necessarily reflect
those of the EU or the ERC Executive Agency.
MB and MV have been partially funded by INdAM-GNCS project
CUP\_E53C23001670001.
LM has been partially funded by MUR (PRIN2022 research grant n. 202292JW3F).
GV and MV have been partially funded by MUR (PRIN2022 research grant n. 2022MBY5JM).
GV has been partially funded
by the Italian Ministry of Universities and Research (MUR)
and the European Union through Next Generation EU, M4.C2.1.1,
through the grant
PRIN2022PNRR n. P2022M7JZW ``SAFER MESH - Sustainable mAnagement
oF watEr Resources ModEls and numerical MetHods'' research grant,
CUP H53D23008930001.
All the authors are also members of the Gruppo Nazionale Calcolo Scientifico-Istituto Nazionale di Alta Matematica (GNCS-INdAM).

%%%%%%%%%%%%%%%%%%%%%%%%%%%%%%%%%%%%%%%%%%%%%%%%%%%%%%%%%%%%%%%%%%%%%%%%%%
{\footnotesize
\bibliography{bibliography.bib}}
\bibliographystyle{plain}

\appendix

%%%%%%%%%%%%%%%%%%%%%%%%%%%%%%%%%%%%%%%%%%%%%%%%%%%%%%%%%%%%%%%%%%%%%%%%%%
\section{Well-posedness of the linear elasticity problem in mixed formulation
with essential boundary conditions} \label{appendix}
%%%%%%%%%%%%%%%%%%%%%%%%%%%%%%%%%%%%%%%%%%%%%%%%%%%%%%%%%%%%%%%%%%%%%%%%%%
We show a stability estimate for the Hellinger--Reissner formulation
of linear elasticity problems with essential boundary conditions.
We shall be able to prove a priori estimates
that are explicit in terms of the ratio
between the diameter and the radius of the
largest ball with respect to which the domain
is star-shaped.
The main result of the appendix is Theorem~\ref{theorem:stability-HR} below.

%%%%
\paragraph*{Strong formulation.}
%%%%
Let~$\Omega$ be a bounded, polyhedral, Lipschitz domain in $\mathbb R^3$,
with boundary~$\partial\Omega$ and outward unit normal vector~$\nbf$.
Given a volumetric force~$\fbf$ in~$\Lbf^2(\Omega)$
and~$\sigmaboldN$ in~$\Hbf^{-\frac12}(\partial \Omega)$,
the linear elasticity problem in mixed formulation
with an inhomogeneous, essential boundary condition reads as follows:
Find a symmetric stress tensor~$\sigmabold$ and a displacement field~$\ubf$ such that
\begin{equation} \label{app:strong-formulation}
\begin{cases}
-\divbf \ \sigmabold = \fbf                & \text{in } \Omega\\
\sigmabold = \Cbb\nablaS \ubf   & \text{in } \Omega\\
\sigmabold \nbf=\sigmaboldN     & \text{on } \partial \Omega.
\end{cases}
\end{equation}
On occasion, we shall be demanding extra regularity on the boundary condition, namely
\begin{equation} \label{app:extra-regularity-essential-condition}
    \sigmaboldN \in \Lbf^2(\partial \Omega) .
\end{equation}
This is needed to derive explicit a priori estimates on the solution to~\eqref{app:strong-formulation}.
In what follows,
\begin{equation} \label{app:geometry}
\begin{split}
& \text{$\hOmega$ denotes the diameter of~$\Omega$;}\\
& \text{$\rhoOmega$ denotes the radius of the largest ball
with respect to which~$\Omega$ is star-shaped.}
\end{split}
\end{equation} 

%%%%
\paragraph*{A Stokes' lifting for the inhomogeneous essential condition.}
%%%%
In order to write a weak formulation for~\eqref{app:strong-formulation},
we construct a lifting of the essential condition~$\sigmaboldN$.
To this aim, given the identity tensor $\Ibb$,
we consider the Stokes problem:
Find  a velocity field~$\ufrak$ and a pressure~$\pfrak$ such that
\begin{equation} \label{app:Stokes-strong}
\begin{cases}
-\divbf(\nablaS \ufrak - \pfrak \ \Ibb) = \zerobf    & \text{in } \Omega \\
\div\ufrak = 0                                                & \text{in } \Omega \\
(\nablaS \ufrak -\pfrak \ \Ibb) \nbf = \sigmaboldN    & \text{on } \partial\Omega\\
\int_\Omega \ufrak = \zerobf , \qquad\qquad
\int_\Omega \nabla\times\ufrak = \zerobf .
\end{cases}
\end{equation}
Introduce $\Htildebfo(\Omega)$ as the subspace of $\Hbf^1(\Omega)$
consisting of fields~$\vfrak$ satisfying the two conditions
in the last line of~\eqref{app:Stokes-strong}.

A weak formulation of~\eqref{app:Stokes-strong} reads
\begin{equation} \label{app:Stokes-weak}
\begin{cases}
\text{Find } (\ufrak, \pfrak) \in \Vfrak\times \Qfrak := \Htildebfo(\Omega) \times L^2(\Omega) \text{ such that}\\
(\nablaS\ufrak,\nablaS\vfrak)_{0,\Omega}
    - (\div\vfrak,\pfrak)_{0,\Omega} = \langle \sigmaboldN,\vfrak\rangle
        & \forall \vfrak \in \Vfrak \\
(\div\ufrak,\qfrak)_{0,\Omega} = 0 
        & \forall \qfrak \in \Qfrak .
\end{cases}
\end{equation}
In order to prove the well-posedness of~\eqref{app:Stokes-weak},
we prove the two following technical results.

\begin{lemma}[Stokes' inf-sup condition] \label{lemma:inf-sup-Stokes}
There exists a positive constant~$\beta_0$,
which only depends on the ratio between~$\hOmega$ and~$\rhoOmega$ in~\eqref{app:geometry}
such that
\[
\inf_{\qfrak\in\Qfrak}
\sup_{\vfrak\in\Vfrak}
\frac{(\div\vfrak,\qfrak)_{0,\Omega}}{\SemiNorm{\vfrak}_{1,\Omega} \Norm{\qfrak}_{0,\Omega}}
\ge \beta_0.
\]
\end{lemma}
\begin{proof}
A lower bound for the constant~$\beta_0$ is given, e.g.,
in \cite[Proposition 1.1]{Botti-Mascotto:2025}
in terms of the so-called Babu\v ska-Aziz inequality \cite{Babuska-Aziz:1972}.
In turns, an estimate for the constant appearing in that inequality
(also known as the generalized Poincar\'e inequality),
which is explicit in terms of the ratio between~$\hOmega$ and~$\rhoOmega$ in~\eqref{app:geometry},
is given in \cite[Corollary 19, point 3]{Guzman-Salgado:2021}.
\end{proof}

\begin{lemma}[Korn's inequality] \label{lemma:coercivity-kernel-Stokes}
There exists a positive constant~$\alpha_0$,
which only depends on the ratio between~$\hOmega$ and~$\rhoOmega$ in~\eqref{app:geometry},
such that
\begin{equation} \label{app:Korn}
\alpha_0  \Norm{\nabla \vfrak}_{0,\Omega}
\le \Norm{\nablaS \vfrak}_{0,\Omega}
\qquad\qquad \forall \vfrak \in \Vfrak.
\end{equation}
\end{lemma}
\begin{proof}
We start by noting the splitting of the gradient into its
symmetric ($\nablaS$) and skew-symmetric ($\nablaSS$) parts,
and the identity
\[
\Norm{\nabla\vfrak}_{0,\Omega}^2
= \Norm{\nablaS\vfrak}_{0,\Omega}^2
  + \Norm{\nablaSS\vfrak}_{0,\Omega}^2 .
\]
It suffices to bound the second term on the right-hand side.
Algebraic computations and the fact that $\vfrak$
belongs to~$\Htildebfo(\Omega)$ (notably the fact
that $\nabla\times\vfrak$ has zero average) imply that
\[
\Norm{\nablaSS\vfrak}_{0,\Omega}
= \frac{1}{\sqrt2} \Norm{\nabla\times\vfrak}_{0,\Omega}
= \frac{1}{\sqrt2} \inf_{\cbf \in \Rbb^3} \Norm{\nabla\times\vfrak - \cbf}_{0,\Omega}.
\]
A Ne\v cas-Lions inequality holds true
using, e.g., \cite[Theorem 3.2]{Duran:2012},
with a constant~$C_{NL,0}$ only depending
on the ratio between~$\hOmega$ and~$\rhoOmega$ in~\eqref{app:geometry} such that
\[
\inf_{\cbf \in \Rbb^3} \Norm{\nabla\times\vfrak - \cbf}_{0,\Omega}
\le C_{NL,0} \Norm{\nabla(\nabla\times\vfrak)}_{-1,\Omega}.
\]
Algebraic computations
as in \cite[Lemma~$3.1$]{Botti-Mascotto:2025}
and manipulations on negative Sobolev norms entail
\[
\Norm{\nabla(\nabla\times\vfrak)}_{-1,\Omega}
\le 2 \Norm{\nabla\times(\nablaS\vfrak)}_{-1,\Omega}
\le 2 \Norm{\nablaS\vfrak}_{0,\Omega}.
\]
Combining the above displays yields the assertion.
\end{proof}

We deduce the well-posedness of problem~\eqref{app:Stokes-weak}.

\begin{proposition} \label{proposition:stability-Stokes}
Problem~\eqref{app:Stokes-weak} is well-posed.
Given~$(\ufrak,\pfrak)$ the solutions to~\eqref{app:Stokes-weak},
there exists a positive constant~$c_L$ such that
\begin{equation} \label{app:stab-Stokes-1}
\SemiNorm{\ufrak}_{1,\Omega} + \Norm{\pfrak}_{0,\Omega}
\le c_L \Norm{\sigmaboldN}_{-\frac12,\partial\Omega}.
\end{equation}
If assumption~\eqref{app:extra-regularity-essential-condition}
holds true, then there exists a positive constant~$\widetilde c_L$
only depending on the ratio between~$\hOmega$ and~$\rhoOmega$ in~\eqref{app:geometry} such that
\begin{equation} \label{app:stab-Stokes-2}
\SemiNorm{\ufrak}_{1,\Omega} + \Norm{\pfrak}_{0,\Omega}
\le \widetilde c_L \Norm{\sigmaboldN}_{0,\partial\Omega}.
\end{equation}
\end{proposition} 
\begin{proof}
Existence and uniqueness of a solution to~\eqref{app:Stokes-weak}
are a consequence of the standard inf-sup theory,
and Lemmas~\ref{lemma:inf-sup-Stokes}
and~\ref{lemma:coercivity-kernel-Stokes}.

As for the stability estimates,
we focus first on the case $\sigmaboldN$ only belongs to $\Hbf^{-\frac12}(\partial D)$.
We take $\vfrak=\ufrak$ and $\qfrak=\pfrak$ in~\eqref{app:Stokes-weak}
and deduce
\[
\Norm{\nablaS \ufrak}_{0,\Omega}^2
= \langle \sigmaboldN,\ufrak \rangle
\le \Norm{\sigmaboldN}_{-\frac12,\partial\Omega} \Norm{\ufrak}_{\frac12,\partial\Omega}.
\]
The standard trace inequality \cite[Theorem~3.10, point (iii)]{Ern-Guermond:2021},
the Poincar\'e inequality \cite[Lemma 3.24]{Ern-Guermond:2021},
and Korn's inequality~\eqref{app:Korn}
entail the existence of a positive~$C$ such that
\[
\Norm{\nablaS \ufrak}_{0,\Omega}
\le C \Norm{\sigmaboldN}_{-\frac12,\partial\Omega}.
\]
We are not able to detect an explicit dependence of~$C$
on the ratio between~$\hOmega$ and~$\rhoOmega$ in~\eqref{app:geometry},
due to the use of the standard trace inequality.
Combining the above bounds yields the bound on $\SemiNorm{\ufrak}_{1,\Omega}$
in~\eqref{app:stab-Stokes-1}.
The bound on $\Norm{\pfrak}_{0,\Omega}$
follows from the standard inf-sup theory.
\medskip

We now prove~\eqref{app:stab-Stokes-2} under
assumption~\eqref{app:extra-regularity-essential-condition}.
We take $\vfrak=\ufrak$ and $\qfrak=\pfrak$ in~\eqref{app:Stokes-weak},
and deduce
\[
\Norm{\nablaS \ufrak}_{0,\Omega}^2
= \langle \sigmaboldN,\ufrak \rangle
\le \Norm{\sigmaboldN}_{0,\partial\Omega} \Norm{\ufrak}_{0,\partial\Omega}.
\]
The continuous trace inequality \cite[Lemma~1.49]{DiPietro-Ern:2011}
(constant $c_{cti}$),
the Poincar\'e inequality \cite[Lemma 3.24]{Ern-Guermond:2021}
(constant~$c_P$),
and Korn's inequality~\eqref{app:Korn}
(constant~$\alpha_0$)
give
\[
\Norm{\nablaS \ufrak}_{0,\Omega}
\le c_{cti} \ c_P \ \alpha_0 
    \Norm{\sigmaboldN}_{0,\partial\Omega}.
\]
The three constants in the display above are explicit
with respect to the ratio between~$\hOmega$ and~$\rhoOmega$ in~\eqref{app:geometry}.
The bound on $\SemiNorm{\ufrak}_{1,\Omega}$
in~\eqref{app:stab-Stokes-1} is proven.
We are left with the bound on $\Norm{\pfrak}_{0,\Omega}$,
which indeed follows from \cite[Theorem 4.2.3]{Boffi-Brezzi-Fortin:2013},
and the fact that $\alpha_0$, $\Cafrak$, and $\beta_0$
are explicit with respect to the ratio between~$\hOmega$ and~$\rhoOmega$ in~\eqref{app:geometry}.
\end{proof}

We are now in a position to define a stable
lifting~$\sigmaboldhat$ of~$\sigmaboldN$.
Given $(\ufrak,\pfrak)$ the solution to~\eqref{app:Stokes-weak},
we introduce the symmetric tensor
\begin{equation} \label{app:sigmaboldbar}
\sigmaboldhat := \nablaS \ufrak - \pfrak \ \Ibb.
\end{equation}
The triangle inequality
and Proposition~\ref{proposition:stability-Stokes} yield
\begin{equation} \label{app:stability-lifting-bcs}
\Norm{\sigmaboldhat}_{0,\Omega}
\le \widetilde c_L \Norm{\sigmaboldN}_{0,\partial\Omega},
\end{equation}
where~$\widetilde c_L$ is the constant in~\eqref{app:stab-Stokes-2}.
Moreover, problem~\eqref{app:Stokes-strong} implies
that~$\sigmaboldhat$ is divergence free.

%%%%
\paragraph*{A weak formulation of~\eqref{app:strong-formulation}.}
%%%%

Let~$\Sigmabold$, $\Vbf$, and~$a(\cdot,\cdot)$ be as in~\eqref{spaces&bfs}.
Given~$\sigmaboldhat$ as in~\eqref{app:sigmaboldbar},
an equivalent weak formulation of~\eqref{app:strong-formulation}
reads as follows:
\begin{equation} \label{app:weak-formulation}
\begin{cases}
\text{Find } \sigmabold_0 \in \Sigmabold 
        \text{ and  } \ubf\in \Vbf \text{ such that}\\
a(\sigmabold_0,\taubold) + (\divbf \taubold, \ubf)_{0,\Omega} 
    =  - a(\sigmaboldhat,\taubold)  
        & \forall \taubold \in \Sigmabold\\
(\divbf \sigmabold_0, \vbf)_{0,\Omega} 
    = - (\fbf,\vbf)_{0,\Omega}        & \forall \vbf \in \Vbf.
\end{cases}
\end{equation}
The solution~$\sigmabold$ to~\eqref{app:strong-formulation}
is given by the sum of $\sigmabold_0$ and $\sigmaboldhat$.
Below, under assumption~\eqref{app:extra-regularity-essential-condition},
we shall discuss the well-posedness of~\eqref{app:weak-formulation},
and derive a priori estimates for~$\sigmabold$,
which are explicit with respect to the ratio
between~$\hOmega$ and~$\rhoOmega$ in~\eqref{app:geometry}.

The well-posedness for the case of natural and mixed natural-essential
boundary conditions is discussed in \cite[Chapter~8]{Boffi-Brezzi-Fortin:2013},
yet without an explicit knowledge of the constants in the a priori estimates.

%%%%
\paragraph*{A technical result.}
%%%%
We prove a result, which is instrumental to infer the
coercivity on the kernel of the bilinear form $a(\cdot,\cdot)$
employed in~\eqref{app:weak-formulation}.

\begin{lemma}\label{lem:devdiv}
There exists a positive constant $C_d$ only depending
on the ratio between~$\hOmega$ and~$\rhoOmega$ in~\eqref{app:geometry} such that
\begin{equation} \label{app:devdiv}
\Norm{\taubold}_{0,\Omega}\le C_d
\left( \Norm{\dev\taubold}_{0,\Omega} 
        + \hOmega\Norm{\divbf\taubold}_{0,\Omega} \right)
\qquad\qquad \forall \taubold \in \Sigmabold.
\end{equation}
\end{lemma}
\begin{proof}
Using \cite[Corollary 19, point 3]{Guzman-Salgado:2021},
there exists a positive constant~$c_{GS}$
only depending on~$\hOmega$ and~$\rhoOmega$ such that
for any~$\taubold$ in~$\Sigmabold$,
there exists~$\vbf$ in~$\Hbf^1(\Omega)\cap \Lbf^2_0(\Omega)$ satisfying
\begin{equation}\label{proof:devdiv1}
\div\vbf=\tr(\taubold) 
\;\text{ and }\;
\SemiNorm{\vbf}_{1,\Omega} 
\le c_{GS} \Norm{\tr(\taubold)}_{0,\Omega}.
\end{equation}
The definition of the $\dev$ operator allows us to write
\[
\int_\Omega \tr(\taubold)^2 = \int_\Omega \tr(\taubold)\div\vbf
= \int_\Omega \taubold:\tr(\gradbf\vbf)\Ibb
= 3 \int_\Omega \taubold:(\gradbf\vbf-\dev(\gradbf\vbf)).
\]
Using the symmetry of~$\taubold$, an integration by parts,
the above identity, the definition of~$\Sigmabold$ in~\eqref{spaces&bfs},
the Cauchy--Schwarz inequality,
the Poincar\'e inequality \cite[Section~3.3]{Ern-Guermond:2021} (constant $c_P$
independent of the ratio between~$\hOmega$ and~$\rhoOmega$ in~\eqref{app:geometry}),
and~\eqref{proof:devdiv1} (constant $c_{GS}$), we get
\[
\begin{aligned}
\Norm{\tr(\taubold)}_{0,\Omega}^2
&= - 3 \int_\Omega \dev\taubold : \nablaS \vbf
    - 3 \int_\Omega (\divbf\taubold)\cdot\vbf   \\
& \le 3 \left(\Norm{\dev \taubold}_{0,\Omega} \Norm{\nablaS \vbf}_{0,\Omega} 
    + \Norm{\divbf\taubold}_{0,\Omega} \Norm{\vbf}_{0,\Omega} \right) \\
&\le (1+c_P) \left(\Norm{\dev\taubold}_{0,\Omega} 
    + \hOmega \Norm{\divbf\taubold}_{0,\Omega}  \right) \SemiNorm{\vbf}_{1,\Omega} \\
& \le (1+c_P)c_{GS} \left(\Norm{\dev\taubold}_{0,\Omega} 
    + \hOmega\Norm{\divbf\taubold}_{0,\Omega} \right)
    \Norm{\tr(\taubold)}_{0,\Omega}.
\end{aligned}
\]
We deduce
\[
\Norm{\tr(\taubold)}_{0,\Omega}
\le (1+c_P)c_{GS} \left( \Norm{\dev\taubold}_{0,\Omega} 
            + \hOmega\Norm{\divbf\taubold}_{0,\Omega} \right) .
\]
The assertion follows using again the definition of the $\dev$ operator:
\[
\Norm{\taubold}_{0,\Omega} \le
\Norm{\dev\taubold}_{0,\Omega} + \Norm{\tr(\taubold)}_{0,\Omega} 
\le 2(1+c_P)c_{GS} \left( \Norm{\dev\taubold}_{0,\Omega}
            + \hOmega\Norm{\divbf\taubold}_{0,\Omega} \right) .
\]
\end{proof}

%%%%
\paragraph*{An inf-sup condition and coercivity on the kernel for the Hellinger--Reissner principle.}
%%%%
We recall an inf-sup condition for the Hellinger--Reissner principle,
whose proof can be found in \cite[Proposition 1.5]{Botti-Mascotto:2025}.

\begin{proposition} \label{proposition:inf-sup-HR}
There exists a positive constant $\betastarz$ only depending on the ratio
of~$\hOmega$ and~$\rhoOmega$ such that
\begin{equation} \label{app:infsup}
\inf_{\vbf\in\Vbf}\ \sup_{\taubold\in\Sigmabold} 
            \frac{(\divbf \taubold, \vbf)_{0,\Omega}}
            {\Norm{\taubold}_{\Sigmabold} \Norm{\vbf}_{\Vbf}}
\ge \hOmega^{-1} \betastarz.
\end{equation}    
\end{proposition}

We define the Hellinger--Reissner kernel as
\begin{equation} \label{definition-Kbb}
\mathbb K := \{ \taubold \in \Sigmabold \mid (\divbf \taubold, \vbf)_{0,\Omega}=0 \quad\forall\vbf\in \Vbf\}.
\end{equation}
\begin{proposition} \label{proposition:coercivity-HR}
The coercivity of $a(\cdot,\cdot)$ on the kernel~$\mathbb K$
\begin{equation} \label{app:coercivity-HR}
\alpha_0^* \Norm{\taubold}_{\Sigmabold}^2
\le a(\taubold,\taubold)
\qquad\qquad \forall \taubold\in\mathbb K
\end{equation}
holds true with coercivity constant~$a_0^*$
independent of the ratio between~$\hOmega$ and~$\rhoOmega$ in~\eqref{app:geometry},
and the Lam\'e parameter~$\lambda$.
\end{proposition}
\begin{proof}
We have
\begin{equation}\label{app:coer}
a(\taubold,\taubold)
\ge (2\mu)^{-1} \Norm{\dev \taubold}_{0,\Omega}^2
\overset{\eqref{app:devdiv}, \eqref{definition-Kbb}}{\ge}
    (2\mu C_d)^{-1} \Norm{\taubold}_{0,\Omega}^2
\overset{\eqref{definition-Kbb}}{=} 
    (2\mu C_d)^{-1} \Norm{\taubold}_{\Sigmabold}^2
\qquad \forall \taubold\in\mathbb K.
\end{equation}
\end{proof}
%%

%%%%
\paragraph*{Well posedness of~\eqref{app:weak-formulation}.}
%%%%
We are in a position to prove the main result of the appendix.

%%%%%
\begin{theorem} \label{theorem:stability-HR}
For every~$\fbf$ in~$\Lbf^2(\Omega)$, $\sigmaboldN$ in~$\Lbf^2(\partial\Omega)$,
and~$\sigmaboldhat$ as in~\eqref{app:sigmaboldbar},
problem~\eqref{app:weak-formulation} admits a unique solution
$(\sigmabold,\ubf)$ in $\Sigmabold\times\Vbf$.
Moreover, there exists a positive constant~$c_{HR}$ depending 
on the ratio between~$\hOmega$ and~$\rhoOmega$ 
in~\eqref{app:geometry}, and~$\mu$,
but independent of~$\lambda$ such that
the solution to problem~\eqref{app:weak-formulation}
satisfies the stability bound
\[
\Norm{\sigmabold}_{\Sigmabold}
+ \Norm{\ubf}_{\Vbf}
\le c_{HR} (\Norm{\fbf}_{0,\Omega} 
    + \Norm{\sigmaboldN}_{0,\partial\Omega}).
\]
If assumption~\eqref{app:extra-regularity-essential-condition} is not valid,
well-posedness can be still proved
but without an explicit dependence of the stability constant
in terms of the geometric properties of~$\Omega$.
\end{theorem}
\begin{proof}
Since the right-hand sides in \eqref{app:weak-formulation} define
linear functionals on $\Sigmabold$ and~$\Vbf$, 
existence and uniqueness of the solution to~\eqref{app:weak-formulation} 
follow from the standard inf-sup theory, \eqref{app:infsup}, and~\eqref{app:coer}.

From \cite[eqs. (4.2.36) and (4.2.37)]{Boffi-Brezzi-Fortin:2013},
the constant in a priori estimates only depend on
the continuity constants of the two bilinear forms
and the right-hand side,
the inf-sup constant~$\betastarz$ in~\eqref{app:infsup},
and the coercivity constant~$\alpha_0^*$ in~\eqref{app:coercivity-HR}.
All such constants are independent of~$\hOmega$ and~$\rhoOmega$;
see Propositions~\ref{proposition:inf-sup-HR}
and~\ref{proposition:coercivity-HR}.
On the other hand, the right-hand side involves the lifting~$\sigmaboldhat$
of the essential boundary condition~$\sigmaboldN$.
Under assumption~\eqref{app:extra-regularity-essential-condition},
we can further use~\eqref{app:stability-lifting-bcs}
and deduce the assertion.
\end{proof}

\end{document}